\documentclass{article}
\usepackage[preprint, nonatbib]{neurips_2019}
\RequirePackage{times}

\usepackage{amssymb,amsmath,amscd,amsfonts,amstext,amsthm,bbm, enumerate, dsfont, mathtools}
\usepackage{array}
\usepackage{nicefrac}
\usepackage{booktabs}
\usepackage{multicol}
\usepackage{subcaption}
\usepackage{algorithm}
\usepackage{graphicx}
\usepackage{varwidth}
\usepackage{hyperref}
\usepackage{import}
\usepackage{caption,subcaption}
\usepackage{algpseudocode}
\usepackage[none]{hyphenat}
\usepackage{footnote}
\usepackage{threeparttable}
\usepackage{ textcomp }
\usepackage{tcolorbox}
\usepackage{multirow}
\usepackage{pifont}
\newcommand{\cmark}{\ding{51}}%
\newcommand{\xmark}{\ding{55}}%
\definecolor{mydarkgreen}{RGB}{39,130,67}
\newcommand{\green}{\color{mydarkgreen}}
\definecolor{mydarkred}{RGB}{192,47,25}
\newcommand{\red}{\color{mydarkred}}

\newcommand{\compactify}{} 

\graphicspath{ {images/} }

\newtheorem{theorem}{Theorem}
\newtheorem{proposition}{Proposition}

\newtheorem{lemma}{Lemma}
\newtheorem{corollary}{Corollary}
\newtheorem{assumption}{Assumption}

\newtheorem{innercustomthm}{Theorem}
\newenvironment{customthm}[1]
  {\renewcommand\theinnercustomthm{#1}\innercustomthm}
  {\endinnercustomthm}
  
\newtheorem{innercustomlem}{Lemma}
\newenvironment{customlem}[1]
  {\renewcommand\theinnercustomlem{#1}\innercustomlem}
  {\endinnercustomlem}
  
  \newtheorem{innercustomcor}{Corollary}
\newenvironment{customcor}[1]
  {\renewcommand\theinnercustomcor{#1}\innercustomcor}
  {\endinnercustomlem}

\DeclareMathOperator*{\argmin}{argmin}

\DeclareMathOperator*{\dom}{dom}

\newcommand{\prox}{\mathop{\mathrm{prox}}\nolimits}
\newcommand{\sumin}{\sum \limits_{i=1}^n}

\newcommand{\sumjm}{\sum_{j=1}^m}
\newcommand{\sumkn}{\sum_{k=1}^n}
\newcommand{\sumkm}{\sum_{k=1}^m}
\newcommand{\avein}{\frac{1}{n}\sum_{i=1}^n}

\newcommand{\avejm}{\frac{1}{m}\sum_{j=1}^m}
\newcommand{\avekn}{\frac{1}{n}\sum_{k=1}^n}
\newcommand{\avekm}{\frac{1}{m}\sum_{k=1}^m}
\newcommand{\proxR}{\prox_{\eta R}}
\newcommand{\proxj}{\prox_{\eta_j g_j}}

\newcommand{\EE}{\mathbb{E}}
\newcommand{\RR}{\mathbb{R}}
\newcommand{\cB}{{\cal B}}
\newcommand{\cL}{{\cal L}}
\newcommand{\cO}{{\cal O}}
\newcommand{\cN}{{\cal N}}

\newcommand{\eqdef}{\stackrel{\mathrm{def}}{=}}

\newcommand{\ind}{\chi} 
\newcommand{\cF}{{\cal F}}
\newcommand{\cM}{{\cal M}}
\newcommand{\cC}{{\cal C}}
\newcommand{\cX}{{\cal X}}
\newcommand{\cY}{{\cal Y}}
\newcommand{\mA}{{\bf A}}
\newcommand{\mB}{{\bf B}}

\newcommand{\mD}{{\bf D}}
\newcommand{\mI}{{\bf I}}
\newcommand{\mW}{{\bf W}}

\providecommand{\Range}[1]{\mbox{Range}\left( #1\right)}

\usepackage[colorinlistoftodos,bordercolor=orange,backgroundcolor=orange!20,linecolor=orange,textsize=scriptsize]{todonotes}

\def\<#1,#2>{\left\langle #1,#2\right\rangle}

\author{Konstantin Mishchenko${}^{1, 2}$ \qquad Peter Richt\'{a}rik${}^{3}$ \\
	${}^1$ CNRS, École Normale Supérieure, Inria \\
	${}^2$ Work done when the author was a student at KAUST \\
	${}^3$ King Abdullah University of Science and Technology (KAUST), Thuwal, Saudi Arabia
}

\begin{document}

\title{A Stochastic Decoupling Method for Minimizing\\ the Sum of Smooth and Non-Smooth Functions}

\maketitle

\begin{abstract}
We consider the problem of minimizing the sum of three convex functions: i) a smooth function $f$ in the form of an expectation or a finite average, ii) a non-smooth function $g$ in the form of a finite average of proximable functions $g_j$, and iii) a proximable regularizer $R$. We design a variance-reduced method which is able to progressively learn the proximal operator of $g$ via the computation of the proximal operator of a single randomly selected function $g_j$ in each iteration only. Our method can provably and efficiently accommodate many strategies for the estimation of the gradient of $f$, including via standard and variance-reduced stochastic estimation, effectively decoupling the smooth part of the problem from the non-smooth part. We prove a number of iteration complexity results, including a general $\cO(\nicefrac{1}{t})$ rate, $\cO(\nicefrac{1}{t^2})$ rate in the case of strongly convex smooth $f$, and several linear rates in special cases, including accelerated linear rate. For example, our method achieves a linear rate for the problem of minimizing a strongly convex function $f$ subject to linear constraints under no assumption on the constraints beyond consistency. When combined with SGD or SAGA estimators for the gradient of $f$, this leads to a very efficient method for empirical risk minimization.  Our method generalizes several existing algorithms, including forward-backward splitting, Douglas-Rachford splitting, proximal SGD, proximal SAGA, SDCA, randomized Kaczmarz and Point-SAGA. However, our method leads to many new specific methods in special cases; for instance,  we obtain the first randomized variant of the Dykstra's method for projection onto the intersection of closed convex sets. 
\end{abstract}

\section{Introduction}

In this paper we address optimization problems of the form 
\begin{align}
 \compactify   \min_{x\in\RR^d} F(x) \eqdef f(x) + \frac{1}{m}\sum \limits_{j=1}^m g_j(x) + R(x), \label{eq:pb_general}
\end{align}
where $f\colon\RR^d \to \RR$ is a smooth convex function,
and $R, g_1,\dotsc, g_m\colon \RR^d \to \RR\cup \{+\infty\}$ are  proper closed convex functions, admitting efficiently computable proximal operators\footnote{The proximal operator of function $R$ is defined as
    $\prox_{\eta R}(x) \eqdef \argmin_{u\in \RR^d} \left\{R(u) + \frac{1}{2\eta}\|u - x\|^2 \right\}.$}.   We also assume throughout that $\dom F \eqdef \{x: F(x)<+\infty\} \neq \emptyset$ and, moreover, that the set of minimizers of \eqref{eq:pb_general}, $\cX^*$, is non-empty.
    
    The main focus of this work is on how the difficult non-smooth term \begin{equation}\label{eq:g_sum} \compactify g(x)\eqdef \frac{1}{m}\sum \limits_{j=1}^m g_j(x)\end{equation} should be treated in order to construct an efficient  algorithm for solving the problem.  We are specifically interested in the case when $m$ is very large, and when the proximal operators of  $g$ and $g+R$ are impossible or prohibitively difficult to evaluate. 
We thus need to rely on splitting approaches which make calls to proximal operators of functions $\{g_j\}$ and $R$ separately.

Existing methods for solving problem~\eqref{eq:pb_general} can efficiently handle the case $m=1$  only \cite{alacaoglu2017smooth}. There were a few attempts to design methods capable of handling the general $m$ case, such as~\cite{allen2017katyusha, pedregosa2019proximal, ryu2017proximal} and~\cite{defazio2016simple}. None of the existing methods  offer a linear rate for non-smooth problem except for random projection. In cases when sublinear rates  are established, the assumptions on the functions $g_j$ are very restrictive. For instance, the results in \cite{allen2017katyusha} are limited to Lipschitz continuous $g_j$ only,   and \cite{defazio2016simple} assumes $g_j$ to be strongly convex. This is very unfortunate because the majority of  problems appearing in popular data science and machine learning applications lack these properties. For instance, if we want to find a minimum of a smooth function over the intersection of $m$ convex sets, $g_j$ will be characteristic functions of sets, which are neither Lipschitz nor strongly convex. 

{\bf Applications.}  There is a long list of  applications of the non-smooth finite-sum problem \eqref{eq:pb_general}, including  convex feasibility~\cite{BauBor:96}, constrained optimization~\cite{Nocedal-Wright-book-2006}, decentralized optimization~\cite{Nedic2009distributed}, support vector machine~\cite{cortes1995support},  Dantzig selector~\cite{candes2007dantzig}, overlapping group Lasso~\cite{yuan2006model}, and Fused Lasso. In Appendix~\ref{ap:applications} we elaborate in detail how these problems  can be mapped to the general problem \eqref{eq:pb_general} (in particular, see Table~\ref{tbl:summary_apps}).

{\bf Variance reduction.} Stochastic variance-reduction methods are a major breakthrough of the last decade, whose success started with  the Stochastic Dual Coordinate Ascent (SDCA) method~\cite{shalev2013stochastic} and the invention of the Stochastic Average Gradient (SAG) method~\cite{sag}. Variance reduction has attracted enormous attention and now its reach covers strongly convex, convex and non-convex~\cite{lei2017non} stochastic problems. Despite being originally developed for finite-sum problems,  variance reduction  was shown to be applicable even to problems with $f$ expressed as a general expectation~\cite{lei2017less, nguyen2018inexact}. Further generalizations and extensions include variance reduction for minimax problems~\cite{palaniappan2016stochastic}, coordinate descent in the general $R$ case~\cite{hanzely2018sega}, and minimization with arbitrary sampling~\cite{gower2018stochastic}.  However, very little is known about variance reduction for non-smooth finite sum problems.

\section{Summary of Contributions}  

The departure point of our work is the observation that there is a class of non-smooth problems for which variance reduction is {\em not} required; these are the linear feasibility problems: given $\mA\in \RR^{m\times d}$ and $b\in \RR^m$, find $x\in \RR^d$ such that $\mA x=b$. Assuming the system is consistent, this problem can be cast as an instance of  \eqref{eq:pb_general}, with $R \equiv 0$ , $f(x)=\frac{1}{2}\|x\|^2$ and $g_j$ corresponding to the characteristic function of the $j$-th equation in the system. Efficient SGD methods (or equivalently, randomized projection methods)  with linear convergence rates were recently developed for this problem \cite{gower2015randomized, richtarik2017stochastic, tu2017breaking}, as well as accelerated variants~\cite{tu2017breaking,richtarik2017stochastic, gower2018accelerated} whose linear rate yields a quadratic improvement in the iteration complexity. However, it is {\em not} known whether these or similar linear rates could be obtained when one considers $f$ to be an arbitrary smooth and strongly convex function. While our work was originally motivated by the quest to answer this question, and  we answer in the affirmative, we were able to build a much more general theory, as we explain below.

We now summarize some of the most important contributions of our work:

{\bf First variance reduction for $g$.}  We propose a {\em  variance-reduction}  strategy for progressively approximating the proximal operator of the average of a large number of {\em non-smooth} functions $g_j$ via only evaluating the proximal operator of a single function $g_j$ in each iteration. That is, unlike existing approaches, we are able to treat the difficult  term \eqref{eq:g_sum} for any $m$. Combined with a gradient-type step in $f$ (we allow for multiple ways in which the gradient estimator is built; more on that below), and a proximal step for $R$, this leads to a {\em new and remarkably efficient method} (Algorithm~\ref{alg:sdm}) for solving problem~\eqref{eq:pb_general}. 

{\bf Compatibility with any gradient estimator for $f$.}
Our variance-reduction scheme for the non-smooth term $g$ is  {\em decoupled} from the way we choose to construct gradient estimators for $f$.  This allows us to use the most efficient and suitable estimators depending on the structure of $f$. In this regard, two cases are of particular importance: i) $f =\EE_\xi f_\xi $, where  $f_\xi \colon \RR^d\to \RR$ is almost surely convex and smooth, and ii) $f	= \frac{1}{n}\sum_i f_i$, where $\{f_i\}$ are convex and smooth.  In case i) one may consider the standard stochastic gradient estimator $\nabla f_{\xi^k}(x^k)$, or a mini-batch variant thereof,  and in case ii) one may consider the batch gradient $\nabla f(x^k)$ if $n$ is small, or a variance-reduced gradient estimator, such as SVRG~\cite{SVRG, L-SVRG} or SAGA~\cite{defazio2014saga, SAGA-AS}, if $n$ is large. {\em Our general analysis allows for any estimator to be used as long as it satisfies a certain technical assumption (Assumption~\ref{as:method}).} In particular, to illustrate the versatility of our approach, we show that this assumption holds for estimators used by Gradient Descent, SVRG, SAGA and over-parameterized SGD. We also claim without a proof that a variant of coordinate descent~\cite{hanzely2018sega} satisfies our assumption, but leave it for future work.

{\bf Future-proof design.} Our analysis is compatible with a wide array of other estimators of the gradient of $f$ beyond the specific ones listed above. Therefore, new specific variants of our generic method for solving problem \eqref{eq:pb_general} can be obtained in the future by marrying any such new  estimators with our variance-reduction strategy for the non-smooth finite sum term $g$.

{\bf Special cases.} Special cases of our method include randomized  Kaczmarz method~\cite{karczmarz1937angenaherte, RK}, Douglas-Rachford splitting \cite{DRsplitting}, forward-backward splitting~\cite{Nesterov_composite2013, FB2011}, a variant of SDCA~\cite{shalev2013stochastic},  and Point-SAGA~\cite{defazio2016simple}. Also, we obtain the first {\em randomized} variant of the famous Dykstra's algorithm~\cite{dykstra1983algorithm} for projection onto the intersection of convex sets. These special cases are summarized in Table~\ref{tab:special_cases}.

\begin{table}[t]
   \centering    
   \caption{Selected special cases of our method. For Dykstra's algorithm, $\cC_1, \dotsc, \cC_m$ are closed convex sets; and we  wish to find projection onto their intersection. Randomized Kaczmarz is a special case for linear constraints (i.e., $\cC_j = \{x: a_j^\top x = b_j\}$). We do not prove convergence under the same assumptions as Point-SAGA as they require strong convexity and smoothness of each $g_j$, but the algorithm is still a special case.}
    \footnotesize
    \begin{tabular}{cccccc}
        \toprule[.1em]
         $f$ & $g_j$ & $R$ & $\eta$ & Method & Comment  \\
         \midrule
         $f_1=f$, $n=1$ & 0 & $R$ & $< \nicefrac{2}{L}$ & Forward-Backward &  \cite{Nesterov_composite2013, FB2011}\\
                 \addlinespace 
         0 & $g_1=g$, $m=1$ & $R$ & any & Douglas-Rachford & \cite{DRsplitting}\\
           \addlinespace                  
         $\EE_\xi f_\xi $ & 0 & $R$ & $\le \nicefrac{1}{4L}$ & Proximal SGD & \cite{duchi2009efficient} \\
         \addlinespace
         $\nicefrac{1}{n}  \sum_i f_i$ & 0 & $R$ & $\le \nicefrac{1}{5L}$ & Proximal SAGA & \cite{defazio2014saga} \\
         \addlinespace
         $\nicefrac{1}{2}\|x - x^0\|^2$ & $g_j$ & 0 & $\eta=\nicefrac{1}{m}$ &  SDCA & \cite{shalev2013stochastic} \\
         \addlinespace
         $\nicefrac{1}{2}\|x - x^0\|^2$ & $\ind_{\cC_j}$ & 0 & $\eta=\nicefrac{1}{m}$ &  Randomized Dykstra's algorithm & NEW \\
         \addlinespace
         $\nicefrac{1}{2}\|x - x^0\|^2$ & $\ind_{\{x:a_j^\top x = b_j\}}$ & 0 & $\eta=\nicefrac{1}{m}$ &  Randomized Kaczmarz method & \cite{karczmarz1937angenaherte, RK} \\
         \addlinespace
         0 & $g_j$ & 0 & any & Point-SAGA & \cite{defazio2016simple} \\
         \addlinespace
         $f_1=f$, $n=1$ & $g_1=g$, $m=1$ & $R$ & $< \nicefrac{2}{L}$ &   Condat-V{\~u} algorithm~& \cite{vu2013splitting, condat2013primal} \\
         \bottomrule[.1em]
    \end{tabular}
    \label{tab:special_cases}
\end{table}

{\bf Sublinear rates.} 
We first prove convergence of the iterates to the solution set in a Bregman sense, without quantifying the rate (see Appendix~\ref{ap:almost_sure}). Next, we establish $\cO\left(\nicefrac{1}{t}\right)$ rate with constant stepsizes under no assumption on problem~\eqref{eq:pb_general} beyond the existence of a solution and a few technical assumptions (see Thm~\ref{th:1_over_t_rate}). The rate improves to $\cO\left(\nicefrac{1}{t^2}\right)$ once we assume strong convexity of $f$, and allow for carefully designed decreasing stepsizes (see Thm~\ref{th:1_t2_rate}).

{\bf Linear rate in the non-smooth case with favourable data.} Consider the special case of \eqref{eq:pb_general} with $f$ being strongly convex, $R\equiv 0$ and $g_j(x) = \phi_j(\mA_j^\top x)$, where $\phi_j:\RR^{d_j}\to \RR\cup \{+\infty\}$ are proper closed convex functions, and $\mA_j\in \RR^{d\times d_j}$ are given (data) matrices: \begin{align}
    \compactify \min_{x\in\RR^d} f(x) + \frac{1}{m} \sum \limits_{j=1}^m \phi_j(\mA_j^\top x). \label{eq:pb_linear}
\end{align}
If the smallest eigenvalue of $\mA^\top \mA$ is positive, i.e., $\lambda_{\min}(\mA^\top \mA)>0$, where $\mA = [\mA_1, \dots, \mA_m]\in \RR^{d\times \sum_j d_j}$, then our method converges linearly (see Thm~\ref{th:lin_conv_lin_model}; and note that this can only happen if $\sum_j d_j \leq d$). Moreover, picking $j$ with probability proportional to $\|\mA_j\|$ is optimal (Cor~\ref{cor:imp_sampl}).
In the special case when $\phi_j(y)=\ind_{\{x\;:\;\mA_j^\top x = b_j\}}(x)$ for some vectors $b_1 \in \RR^{d_1},\dots, b_m\in \RR^{d_1}$, i.e., if we are minimizing a strongly convex function under a linear constraint,
\[\min_{x\in \RR^d} \left\{f(x) \;:\; \mA^\top x = b \right\},\]
then the rate is linear even if $\mA^\top \mA$ is not positive definite\footnote{By $\ind_{\cC}(x)$ we denote the characteristic function of the set $\cC$, defined as follows: $\ind_{\cC}(x) = 0$ if $x\in \cC$ and $\ind_{\cC}(x) = +\infty$ if $x\notin \cC$}. The rate will depend on $\lambda_{\min}^+(\mA^\top \mA)$, i.e., the smallest positive eigenvalue (see Thm~\ref{th:lin_constr}).

{\bf Linear and accelerated rate in the smooth case.}
If $g_1, \dotsc, g_m$ are smooth functions, the rate is linear (see Thm~\ref{th:lin_conv_smooth}). If $m$ is big enough, then it is also {\em accelerated} (Cor~\ref{cor:acc_in_g}).  A summary of our iteration complexity results is provided in Table~\ref{tab:results}.

\begin{table}[t]
    \centering
    \caption{Summary of iteration complexity results that we proved. We assume by default that all functions are convex, but provide different rates based on whether $f$ is strongly convex (scvx) and whether $g_1, \dotsc, g_m$ are smooth functions, which is represented by the check marks.}
    \footnotesize
    \begin{tabular}{cccccc}
        \toprule[.1em]
         Problem & $f$ scvx & $g_j$ smooth & Method for $f$ & Rate & Theorem \\
         \midrule
         \multirow{2}{*}{$\EE f_\xi(x) + \frac{1}{m}\sum\limits_{j=1}^m g_j(x) + R(x)$}  & \centering {\large\red\xmark} & {\large\red\xmark} &
\multirow{3}{*}{SGD}     & \large{$\cO\left( \nicefrac{1}{\sqrt{t}}\right)$} & Cor.~\ref{cor:sgd}\\[0.5ex]
         \cline{2-3} \cline {5-6}
         & \centering {\large\green\cmark} & {\large\red\xmark} &   &{\large$\cO\left(\nicefrac{1}{t}\right)$}  & \ref{th:sgd_str_cvx} \\[0.5ex]
         \hline
         \multirow{4}{*}{$\frac{1}{n}\sum\limits_{i=1}^n f_i(x) + \frac{1}{m}\sum\limits_{j=1}^m g_j(x) + R(x)$}  & \centering {\large\red\xmark} & {\large\red\xmark} &   \multirow{6}{*}{\shortstack{GD,\\ SVRG\\ and\\ SAGA}}  
         & {\large$\cO\left( \nicefrac{1}{t} \right)$} & \ref{th:1_over_t_rate} \\[0.5ex]
         \cline{2-3} \cline {5-6}
         & \centering {\large\green\cmark} & {\large\red\xmark} &   & {\large$\cO\left( \nicefrac{1}{t^2} \right)$} & \ref{th:1_t2_rate} \\[0.5ex]
         \cline{2-3} \cline {5-6}
         & \centering {\large\green\cmark} & {\large\green\cmark} &   & Linear & \ref{th:lin_conv_smooth} \\[0.5ex]
         \cline{1-3} \cline {5-6}
         $\frac{1}{n}\sum\limits_{i=1}^n f_i(x) + \frac{1}{m}\sum\limits_{j=1}^m \phi_j(\mA_j^\top x)$ & \centering {\large\green\cmark} & {\large\red\xmark} & & Linear & \ref{th:lin_conv_lin_model}, \ref{th:lin_constr}\\[0.5ex]
         \bottomrule[.1em]
    \end{tabular}
    \label{tab:results}
\end{table}


{\bf Related work.} The problems that we consider recently received a lot of attention. However, we are the first to show linear convergence on non-smooth problems. $\cO\left(\nicefrac{1}{t}\right)$ convergence with stochastic variance reduction was obtained in~\cite{ryu2017proximal} and~\cite{pedregosa2019proximal}, although both works do not have $\cO\left(\nicefrac{1}{t^2}\right)$ rate as we do. On  the other hand, works such as~\cite{yurtsever2016stochastic,cevher2018stochastic} managed to prove $\cO\left(\nicefrac{1}{t^2}\right)$ convergence, but only with all functions from $f$ and $g$ used at every iteration. Stochastic $\cO\left(\nicefrac{1}{t^2}\right)$ for constrained minimization can be found in~\cite{mishchenko2018stochastic}. There is also a number of works that consider parallel~\cite{deng2017parallel} ($\cO\left(\nicefrac{1}{t}\right)$ rate) and stochastic~\cite{zheng2016fast, liu2017accelerated} variants of ADMM, which work with one non-smooth term composed with a linear transformation. To show linear convergence they require matrix in the transformation to be positive-definite. variance-reduced ADMM for compositions, which is an orthogonal direction to ours, was considered in~\cite{yu2017fast}. There is a method for non-smooth problems with $f\equiv 0$ and proximal operator preconditioning that was analyzed in detail in~\cite{chambolle2018stochastic}, we discuss the relation to it in Appendix~\ref{ap:spdhg}. Many methods were designed to work with non-smooth functions  in parallel only, and one can obtain more of them from three-operator splitting methods such as the Condat-V{\~u} algorithm~\cite{vu2013splitting, condat2013primal}. Several works obtained linear convergence for smooth $g$~\cite{du2018linear, palaniappan2016stochastic}. Coordinate descent methods for two non-smooth functions were considered in~\cite{alacaoglu2017smooth}.

\section{Preliminaries}
{\bf Convexity and smoothness.}  A  differentiable function $f:\RR^d\to \RR$ is called {\em $\mu$-strongly convex} if 
$  f(x)       \ge f(y) + \<\nabla  f(y), x - y> + \frac{\mu}{2}\|x - y\|^2$ for all  $x, y\in \RR^d$. It is called {\em convex}   if this holds with $\mu=0$.
A convex function $f:\RR^d\to \RR$ is called {\em $L$-smooth} if it is differentiable and satisfies $f(x) \le f(y) + \<\nabla f(y), x - y> + \frac{L}{2}\|x - y\|^2$ for all $ x, y\in \RR^d.$ 

{\bf Bregman divergence.} To simplify the notation and proofs, it is convenient to work with Bregman divergences. The Bregman divergence associated with a differentiable convex function $f$ is the function
$D_f(x, y) \eqdef f(x) - f(y) - \<\nabla f(x), x - y>. $
It is important to note that the {\em Bregman divergence} of a convex function is always non-negative and is a (non-symmetric) notion of ``distance'' between $x$ and $y$. For $x^*\in \cX^*$, the quantity $D_f(x, x^*)$ serves as a generalization of the functional gap $f(x) - f(x^*)$ in cases when  $\nabla f(x^*)\neq 0$.

Useful inequalities related to convexity, strong convexity and smoothness are summarized in Appendix~\ref{sec:basic_facts}. We will make the following assumption related to optimality conditions.
\begin{assumption}\label{as:optimality}
    There exists $x^*\in \cX^*$ and vectors $y_1^*\in\partial g_1(x^*), \dots, y_m^*\in \partial g_m(x^*)$ and $r^*\in \partial R(x^*)$ such that
$
        \nabla f(x^*) + \avejm y_j^* + r^* = 0.
$
\end{assumption}
Throughout the paper, we will assume that some $x^*$ and $y_1^*, \dotsc, y_m^*$ satisfying Assumption~\ref{as:optimality} are fixed and all statements relate to these objects. We will denote $y^*\eqdef \frac{1}{m}\sumjm y_j^*$. 
A commentary and further details related to this assumption can be found in Appendix~\ref{sec:Opt_Cond}.

\section{The Algorithm}

\begin{algorithm}[t]
   \caption{Stochastic Decoupling Method (SDM).}
   \label{alg:sdm}
\begin{algorithmic}[1]
   \Require Stepsize $\eta$, initial vectors $x^0$, $y_1^0, \dotsc, y_m^0\in \RR^d$, probabilities $p_1,\dotsc, p_m$, oracle that gives gradient estimates
   \For{$t=0,1,\dotsc$}
	   \State Produce an estimate $v^t$ of $\nabla f(x^t)$, e.g., $v^t=\nabla f(x^t)$ 
	   \State $y^t = \avekm y_k^t$
	   \State $z^{t} = \proxR(x^t - \eta v^t - \eta y^t)$
	   \State Sample $j$ from $\{1,\dotsc, m\}$ with probabilities $\{p_1, \dotsc, p_m\}$ and set $\eta_j = \frac{\eta}{mp_j}$
	   \State $x^{t+1} = \proxj(z^t + \eta_j y_j^t)$
	   \State $y_j^{t+1} = y_j^t + \frac{1}{\eta_j}(z^t - x^{t+1})$ \Comment{$y_j^{t+1}\in \partial g_j(x^{t+1})$}
   \EndFor
\end{algorithmic}
\end{algorithm}

Our method  is very general and can work with different types of gradient update. One only needs to have for each $x^t$ an estimate of the gradient $v^t$ such that $\EE v^t = \nabla f(x^t)$ plus an additional assumption about its variance. We also maintain an estimate $y^t $ of full proximal step with respect to $g$, which allows us to make an intermediate step
$
    z^{t} 
    = \prox_{\eta R}(x^t - \eta v^t - \eta y^t).
$
The key idea of this work is then to combine it with variance reduction in the non-smooth part. In fact, it mimics variance-reduction step from~\cite{defazio2016simple}, which was motivated by the SAGA algorithm~\cite{defazio2014saga}. Essentially, the expression above for $z^t$ does not allow for update of $y^t$, so we do one more step,
\begin{align*}
	x^{t+1} =  \proxj(z^{t} + \eta_j y_j^t).
\end{align*}
This can additionally be rewritten using the identity $\prox_{\eta g}(x) \in x - \eta \partial g(\prox_{\eta g}(x))$ as
\begin{align*}
    x^{t+1} 
    &\in x^t - \eta (v^t + \partial R(z^t) + y^t) - \eta_j(\partial g_j(x^{t+1}) - y_j^t)  \approx \prox_{\eta (R+g)}(x^t - \eta \nabla f(x^t)).
\end{align*}
To make sure that the approximation works, we want to make $y_j^t$ be close to $\partial g_j(x^{t+1})$, which we do not know in advance. However, we do it in hindsight by updating $y_j^{t+1}$ with a particular subgradient from $\partial g_j(x^{t+1})$, namely 
$
	y_j^{t+1}
	 = \frac{1}{\eta_j}(z^t + \eta_j y_j^t - \proxj(z^t + \eta_j y_j^t))\in \partial g_j(x^{t+1}).
$

We also need to accurately estimate $\nabla f(x^t)$, and there several options for this. The simplest choice is setting $v^t = \nabla f(x^t)$. Often this is too expensive and one can instead construct $v^t$ using a variance-reduction technique, such as SAGA~\cite{defazio2014saga} (see Algorithm~\ref{alg:v_saga}). To a reader  familiar with Fenchel duality, it might be of some interest that there is an explanation of our ideas using the dual.\footnote{Indeed, problem~\eqref{eq:pb_general} can be recast into
$\min_x \max_{y_1,\dotsc, y_m} f(x) + R(x) + \frac{1}{m}\sum_{j=1}^m x^\top y_j - \frac{1}{m} \sumjm g_j^*\left(y_j\right),
$
where $g_j^*$ is the Fenchel conjugate of $g_j$. Then, the proximal gradient step in $x$ would be
$
	z 
	= \prox_{\eta R}\left(x - \eta\nabla f(x) - \eta \frac{1}{m} \sumjm y_j \right).
$
In contrast, our update in $y_j$ is a proximal block-coordinate ascent step, so the overall process is akin to proximal alternating gradient descent-ascent (see~\cite{bianchi2015coordinate, combettes2015stochastic} for related ideas). However, this is neither how we developed nor analyze the method, so this should not be seen as a formal explanation.} 

\section{Gradient Estimators}


Since we want to have analysis that puts many different methods under the same umbrella, we need an assumption that is easy to satisfy. In particular, the following will fit our needs.
\begin{assumption}\label{as:method}
	Let $w^t \eqdef x^t - \eta v^t$ and $w^* \eqdef x^* - \eta \nabla f(x^*)$. We assume that the oracle produces $v^t$ and (potentially) updates some other variables in such a way that for some constants $\eta_0>0$, $\omega > 0$ and non-negative sequence $\{\cM^t\}_{t=0}^{+\infty}$, such that the following holds for any $\eta \le \eta_0$:
	\begin{enumerate}[(a)]
		\item If $f$ is convex, then 
		$			\EE\|w^{t} - w^*\|^2 + \cM^{t+1}
			\le \EE\|x^t - x^*\|^2 - \omega\eta\EE D_f(x^t, x^*) + \cM^t.
		$
		\item If $f$ is $\mu$-strongly convex, then  either $\cM^t=0$ for all $t$ or there exists $\rho > 0$ such that
		\begin{align*}
			\EE\|w^{t} - w^*\|^2 + \cM^{t+1}
			\le (1 -\omega\eta\mu)\EE\|x^t - x^*\|^2 + (1 - \rho)\cM^t.
		\end{align*}
	\end{enumerate}
\end{assumption}
We note that we could easily make a slightly different assumption to allow for a strongly convex $R$, but this would be at the cost of analysis clarity. Since the assumption above is already quite general, we choose to stick to it and claim without a proof that in the analysis it is possible to transfer strong convexity from $f$ to $R$.

Another observation is that part~(a) of Assumption~\ref{as:method} implies its part~(b) with $\omega/2$. However, to achieve tight bounds for Gradient Descent we need to consider them separately.

\begin{lemma}[Proof in Appendix~\ref{ap:gd}]\label{lem:gd}
	If $f$ is convex, the Gradient Descent estimate, $v^t=\nabla f(x^t)$, satisfies Assumption~\ref{as:method}(a) with any $\eta_0 < \nicefrac{2}{L}$, $\omega = 2 - \eta_0 L$ and $\cM^t = 0$. If $f$ is $\mu$-strongly convex, Gradient Descent satisfies Assumption~\ref{as:method}(b) with $\eta_0 = \frac{2}{L + \mu}$, $\omega = 1$ and $\cM^t=0$.
\end{lemma}
Since $\cM^t=0$ for Gradient Descent, one can ignore $\rho$ in the convergence results or treat it as $+\infty$.

\begin{lemma}[Proof in Appendix~\ref{ap:svrg_saga}]\label{lem:svrg_saga}
	In SVRG and SAGA, if $f_i$ is $L$-smooth and convex for all $i$, Assumption~\ref{as:method}(a) is satisfied with $\eta_0 = \nicefrac{1}{6L}$, $\omega = \nicefrac{1}{3}$ and 
$	\cM^t 
		= \frac{3\eta^2}{n} \sum_i \EE\|\nabla f_i(u_i^t) - \nabla f_i(x^*)\|^2,
$
	where in SVRG $u_i^t=u^t$ is the reference point of the current loop, and in SAGA $u_i^t$ is the point whose gradient is stored in memory for function $f_i$. If $f$ is also strongly convex, then Assumption~\ref{as:method} holds with $\eta_0 = \nicefrac{1}{5L}$, $\omega = 1$, $\rho = \nicefrac{1}{3n}$ and the same $\cM^t$.
\end{lemma}

\begin{algorithm}[t]
   \caption{SAGA Oracle}
   \label{alg:v_saga}
\begin{algorithmic}[1]
	\Require $x^t$, table of past gradients $\nabla f_1(u_1^t), \dotsc, \nabla f_n(u_n^t)$ and their average $\alpha^t$
	\State Sample subset $S$ from $\{1,\dotsc, n\}$ of size $\tau$
	\State $v^t= \frac{1}{\tau}\sum_{i\in S}\left(\nabla f_i(x^t) - \nabla f_i(u_i^t)\right) + \alpha^t$
	\State For all $i\in S$ update  $\nabla f_i(u_i^{t+1})$ with $u_i^{t+1} = x^t$
	\\ \Return $v^t$
\end{algorithmic}
\end{algorithm}

\begin{lemma}[Proof in Appendix~\ref{ap:sgd}]\label{lem:sgd}
	Assume that at an optimum $x^*$ the variance of stochastic gradients is finite, i.e., $\sigma_*^2\eqdef \EE_\xi \|\nabla f_\xi(x^*) - \nabla f(x^*)\|^2 < +\infty$. Then, SGD that terminates after at most $t_0$ iterations satisfies Assumption~\ref{as:method}(a) with $\eta_0=\frac{1}{4L}$, $\omega=1$ and $\rho=0$. In this case, sequence $\{\cM^t\}_{t=0}^{t_0}$ is given by
$
		\cM^t = 2\eta^2(t_0 - t)\sigma_*^2.
$
	If $f$ is strongly convex and $\sigma_*=0$, it satisfies Assumption~\ref{as:method}(b) with $\eta_0 = \frac{1}{2L}$, $\omega=1$ and $\cM^t=0$.
\end{lemma}

There are two important cases for SGD. If the model is overparameterized, i.e., $\sigma_*\approx 0$,  we get almost the same guarantees for SGD as for GD. If,  $\sigma_*\gg 0$, then one needs to choose $\eta = \cO\left(\nicefrac{1}{(\sqrt{t_0} L)}\right)$ in order to keep $\cM^0$ away from $+\infty$. This effectively changes the  $\cO\left(\nicefrac{1}{t}\right)$ rate to $\cO\left(\nicefrac{1}{\sqrt{t}}\right)$, see Cor~\ref{cor:sgd}. Moreover,  obtaining a $\cO\left(\nicefrac{1}{t}\right)$ rate for strongly convex case requires a separate proof.

\section{Convergence}
Let  $\gamma \eqdef \min_{j=1,\dotsc,m} \frac{1}{\eta_j L_j}$, where $L_j\in\RR\cup \{+\infty\}$ is the smoothness constant of $g_j$, in most cases giving $L_j=+\infty$ and $\gamma=0$. Tho goal of our analysis is to show that with introducing new term in the Lyapunov function,
$
	\cY^t
	\eqdef (1+\gamma)\sumkm \eta_k^2\EE\|y_k^{t} - y_k^*\|^2,
$
the convergence is not significantly hurt. This term will be always incorporated in the full Lyapunov function defined as
\begin{align*}
        \cL^t
        \eqdef  \EE\|x^t - x^*\|^2 + \cM^t + \cY^t,
\end{align*}
where $\cM^t$ is from Assumption~\ref{as:method}. In the proof of $\cO\left(\nicefrac{1}{t^2}\right)$ rate we will use decreasing stepsizes and $\cY^t$ will be defined slightly differently, but except for this, it is going to be the same Lyapunov function everywhere.

\subsection{$\cO(\nicefrac{1}{t})$ convergence for general convex problem}

\begin{theorem}[Proof in Appendix~\ref{ap:1_t_rate}]\label{th:1_over_t_rate}
    Assume $f$ is $L$-smooth and $\mu$-strongly convex, $g_1, \dotsc, g_m, R$ are proper, closed and convex. If we use a method for generating $v^t$ which satisfies Assumption~\ref{as:method} and $\eta\le \eta_0$, then
\[
 \compactify 
        \EE D_f(\overline x^t, x^*)
        \le \frac{1}{\omega \eta t}\cL^0,
\]
    where $\cL^0\eqdef \|x^0 - x^*\|^2 + \cM^0 + \sumkm  \eta_k^2 \|y_k^0 - y_k^*\|^2$ and $\overline x^t \eqdef \frac{1}{t}\sum_{k=0}^{t-1} x^k$.
\end{theorem}

If $R\equiv 0$ and $g_j\equiv 0$ for all $j$, then this transforms into $\cO(\nicefrac{1}{t})$ convergence of $f(x^t) - \min f(x)$, which is the correct rate.

The next result takes care of the case when SGD is used, which requires special consideration.
\begin{corollary}\label{cor:sgd}
	If we use SGD for $t$ iterations with constant stepsize, the method converges to a neighborhood of radius $\nicefrac{\cM^0}{\eta t}=2\eta\sigma_*^2$. If we choose the stepsize $\eta = \Theta\left(\nicefrac{1}{(L\sqrt{t})} \right)$, then $2\eta\sigma_*^2=\cO(\nicefrac{1}{\sqrt{t}})$, and we recover $\cO\left(\nicefrac{1}{\sqrt{t}}\right)$ rate.
\end{corollary}

\subsection{$\cO(\nicefrac{1}{t^2})$ convergence for strongly convex $f$}
In this section, we consider a variant of Algorithm~\ref{alg:sdm} with time-varying stepsizes,
\begin{align*}
	z^t 
	&= \prox_{\eta^{t+1} R}(x^t - \eta^{t} v^t - \eta^{t} y^t), \qquad x^{t+1} 
	= \prox_{\eta_j^{t} g_j}(z^t + \eta_j^{t} y_j^t).	
\end{align*}

\begin{theorem}[Proof in Appendix~\ref{ap:1_t2_rate}]\label{th:1_t2_rate}
	Consider updates with time-varying stepsizes, $\eta^t = \frac{2}{\mu\omega(a + t)}$ and $\eta_j^t = \frac{\eta^t}{m p_j}$ for $j=1,\dotsc, m$, where $a\ge 2\max\left\{\frac{1}{\omega\mu\eta_0}, \frac{1}{\rho} \right\}$. Then
\[ \compactify 
		\EE \|x^t - x^*\|^2
		\le \frac{a^2}{(t+a-1)^2}\cL^0,
\]
	where $\cL^0 = \|x^0 - x^*\|^2 + \cM^0 + \sumkm (\eta_k^0)^2 \|y_k^0 - y_k^*\|^2$.
\end{theorem}
This improves upon $\cO(\nicefrac{1}{t})$ convergence proved in~\cite{defazio2016simple} under similar assumptions and matches the bound in~\cite{chambolle2018stochastic}.

In Cor~\ref{cor:sgd} we obtained $\cO(\nicefrac{1}{\sqrt{t}})$ rate for SGD with $\sigma_*\neq 0$. It is not surprising that the rate is worse as it is so even with $g\equiv 0$. For standard SGD we are able to improve the guarantee above to $\cO(\nicefrac{1}{t})$ when the objective is strongly convex.

\begin{theorem}[Proof in Appendix~\ref{ap:sgd_str_cvx}]\label{th:sgd_str_cvx}
	Assume $f$ is $\mu$-strongly convex, $f_\xi$ is almost surely convex and $L$-smooth. Let the update be produced by SGD, i.e., $v^t = \nabla f_{\xi^t}(x^t)$, and let us use time-varying stepsizes $\eta^{t-1} = \frac{2}{a + \mu t}$ with $a\ge 4L$. Then
\[ \compactify 
		\EE \|x^t - x^*\|^2
		\le \frac{8\sigma_*^2}{\mu(a + \mu t)} + \frac{a^2}{(a + \mu t)^2}\cL^0.
\]
\end{theorem}

\subsection{Linear convergence for linear non-smoothness} \label{sec:linear_non_smoothness}

We now provide two linear convergence rates in the case when $R\equiv 0$ and $g_j(x) = \phi_j(\mA_j^\top x)$. 

\begin{theorem}[Proof in Appendix~\ref{ap:lin_conv_lin_model}]\label{th:lin_conv_lin_model}
    Assume that $f$ is $\mu$-strongly convex, $R\equiv 0$, $g_j(x) = \phi_j(\mA_j^\top x)$ for $j=1,\dotsc, m$ and take a method satisfying Assumption~\ref{as:method} with $\rho>0$. Then, if $\eta\le \eta_0$,
\[
 \compactify 
        \EE \|x^t - x^*\|^2         \le \left(1 - \min\{\rho, \omega\eta\mu, \rho_{A}\} \right)^t \cL^0,
\]
    where $\rho_{A} \eqdef \lambda_{\min}(\mA^\top\mA) \min_j \left(\nicefrac{p_j}{\|\mA_j\|}\right)^2$, and $\cL^0\eqdef \|x^0 - x^*\|^2 + \cM^0 + \sumkm  \eta_k^2 \|y_k^0 - y_k^*\|^2$.
\end{theorem}

\begin{corollary}\label{cor:imp_sampl}
    If oracle from Algorithm~\ref{alg:v_saga} (SAGA) is used with probabilities $p_j \propto \|\mA_j\|$, then to get $\EE\|x^t - x^*\|^2 \le \varepsilon$, it is enough to run it for \[ \compactify  \cO\left(\left(n+ \frac{L}{\mu}+ \frac{\|\mA\|_{2, 1}^2}{\lambda_{\min} (\mA^\top \mA)}\right)\log\frac{1}{\varepsilon}\right)\] iterations.
\end{corollary}
Now let us show that this can be improved to depend only on positive eigenvalues if the problem is linearly constrained.
\begin{theorem}[Proof in 	Appendix~\ref{ap:lin_constr}]\label{th:lin_constr}
	Under the same assumptions as in Thm~\ref{th:lin_conv_lin_model} and assuming, in addition, that $g_j = \ind_{\{x: \mA_j^\top x = b_j\}}$ it holds $\EE \|x^t - x^*\|^2 \le (1 - \min\{\rho, \omega\eta\mu, \rho_A\})^t\cL^0$ with $\rho_A=\lambda_{\min}^+(\mA^\top\mA)\min_{j}\left(\nicefrac{p_j}{\|\mA_j\|}\right)^2$, i.e., $\rho_A$ depends only on  the smallest positive eigenvalue of $\mA^\top\mA$.
\end{theorem}
One implication of Thm~\ref{th:lin_constr} is that just by taking a solver such as SVRG we immediately obtain a method for decentralized optimization that will converge linearly. Furthermore, if the problem is ill-conditioned or the communication graph is well conditioned, the leading term is still $\nicefrac{L}{\mu}$, meaning that the rate for decentralized method is the same as for centralized up to constant factors. In Appendix~\ref{ap:lin_constr}, we also give a version of our method specialized to the linearly constrained problem that requires only one extra vector, $y^t$.

\subsection{Linear convergence if all $g_j$ are smooth}
\begin{theorem}[Proof in Appendix~\ref{ap:lin_conv_smooth}]\label{th:lin_conv_smooth}
    Assume that $f$ is $L$-smooth and $\mu$-strongly convex, $g_j$ is $L_j$-smooth for all $j$, Assumption~\ref{as:method}(b) is satisfied and $\eta\le \eta_0$. Then, Algorithm~\ref{alg:sdm} converges as
\[
		\EE \|x^t - x^*\|^2
		\le \left(1 - \min\left\{\omega\eta\mu, \rho, \frac{\gamma}{m(1+\gamma)}\right\}\right)^t\cL^0,
\]   
    where $\gamma \eqdef \min_{j=1,\dotsc,m} (\eta_j L_j)^{-1}$.
\end{theorem}

Based on the theorem above, we suggest to choose probabilities $p_j$ to maximize $\gamma$, which can be done by using $p_j\propto L_j$. If $p_j = \frac{L_j}{\sum_{k=1}^m L_k}$, then $\gamma = \min_{j=1,\dotsc,m} \frac{m p_j}{\eta L_j} = \frac{1}{\eta \overline L}$ with $\overline L\eqdef \frac{1}{m}\sum_{j=1}^m L_j $.
\begin{corollary}[Proof in Appendix~\ref{ap:acc_in_g}]\label{cor:acc_in_g}
	Choose as solver for $f$ SVRG or SAGA without minibatching, which satisfy Assumption~\ref{as:method} with $\eta_0=\nicefrac{1}{5L}$ and $\rho = \nicefrac{1}{3n}$, and consider for simplicity situation where $L_1= \dotsb = L_m \eqdef L_g$ and $p_1=\dotsb=p_m$. Define $\eta_{best} \eqdef (\omega \mu m L_g)^{-\nicefrac{1}{2}}$,
	and set the stepsize to  $\eta=\min\{\eta_0, \eta_{best}\}$. Then the complexity to get $\EE\|x^t - x^*\|^2\le \varepsilon$ is
\[
 \compactify 
		\cO\left(\left(n + m + \frac{L}{\mu} + \sqrt{\frac{mL_g}{\mu}}\right) \log\frac{1}{\varepsilon}\right).
\]
\end{corollary}
Notably, the rate in Cor~\ref{cor:acc_in_g} is accelerated in $g$, suggesting that the proposed update is in some cases optimal. Moreover, if $m$ becomes large, the last term is dominating everything else meaning that acceleration in $f$ might not be needed at all.

\section{Implementation Details and Experiments}

{\bf Randomly generated linear system.}
In this experiment, we first generate a matrix with independent Gaussian entries of zero mean and scale $\nicefrac{1}{\sqrt{d}}$, where $d=100$, and after that we set $\mW\in\RR^{d\times d}$ to be the product of the generated matrix with itself plus identity matrix with coefficient $10^{-2}$ to make sure $\mW$ is positive definite. We also generated a random vector $x^*\in\RR^d$ and took $b = \mW x^*$. The problem is to solve $\mW x=b$, or, equivalently, to minimize $\|\mW x - b\|^2$. We made this choice because it makes estimation of the parameters of accelerated Sketch-and-Project easier.

To run our method, we choose \[f(x) = \frac{1}{2}\|x\|^2\] and \[g_j(x) = \ind_{ \{x: w_j^\top x = b_j\}}(x), \quad j=1,\dotsc, d,\] where $\ind_{\{x \;:\; w_j^\top x=b_j\}}(x)$ is the characteristic function, whose value is $0$ if $w_j^\top x=b_j$ and $+\infty$ otherwise. Then, the proximal operator of $g_j$ is the projection operator onto the corresponding constraint. We found that the choice of stepsize is important for fast convergence and that the value approximately equal $1.3\cdot 10^{-4}\ll 1 = \nicefrac{2}{(L + \mu)}$ led to the best performance for this matrix. 

We compare our method to the accelerated Sketch-and-Project method of~\cite{gower2018accelerated} using optimal parameters. The other method that we consider is classic Kaczmarz method that projects onto randomly chosen constraint. We run all methods with uniform sampling.

\textbf{Linear regression with linear constraints.} We took A9a dataset from LIBSVM and ran $\ell_2$-regularized linear regression, using first 50 observations of the dataset as tough constraints. We compare iteration complexity to precise projection onto all constraints and observe that it takes almost the same number of iterations, although stochastic iterations are significantly cheaper. For each method we chose minibatch of size 20 and stepsizes of order $\nicefrac{1}{L}$ for all methods.

More experiments are provided in Appendix~\ref{sec:add_experiments}.
\begin{figure}
\center
	\includegraphics[scale=0.25]{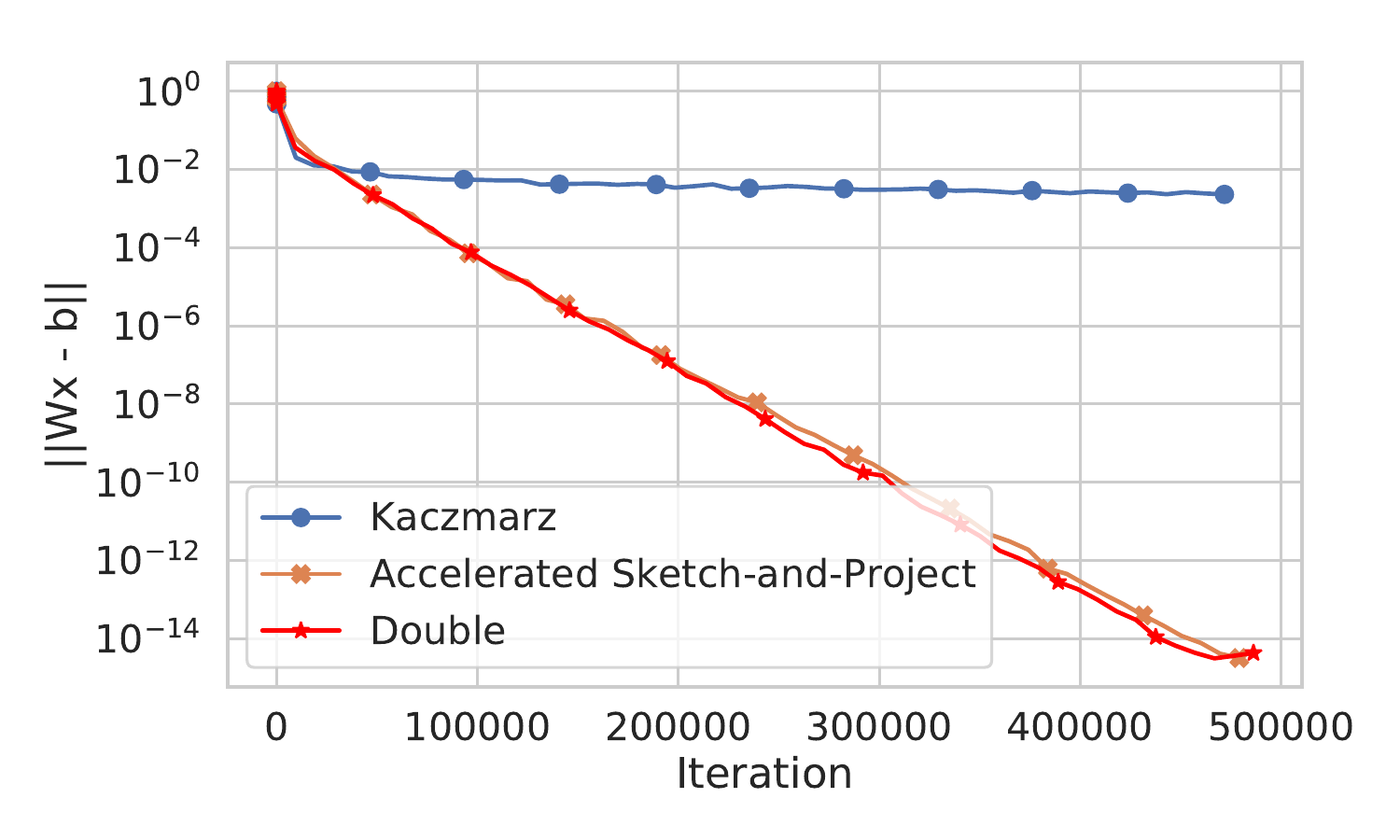}
	\includegraphics[scale=0.23]{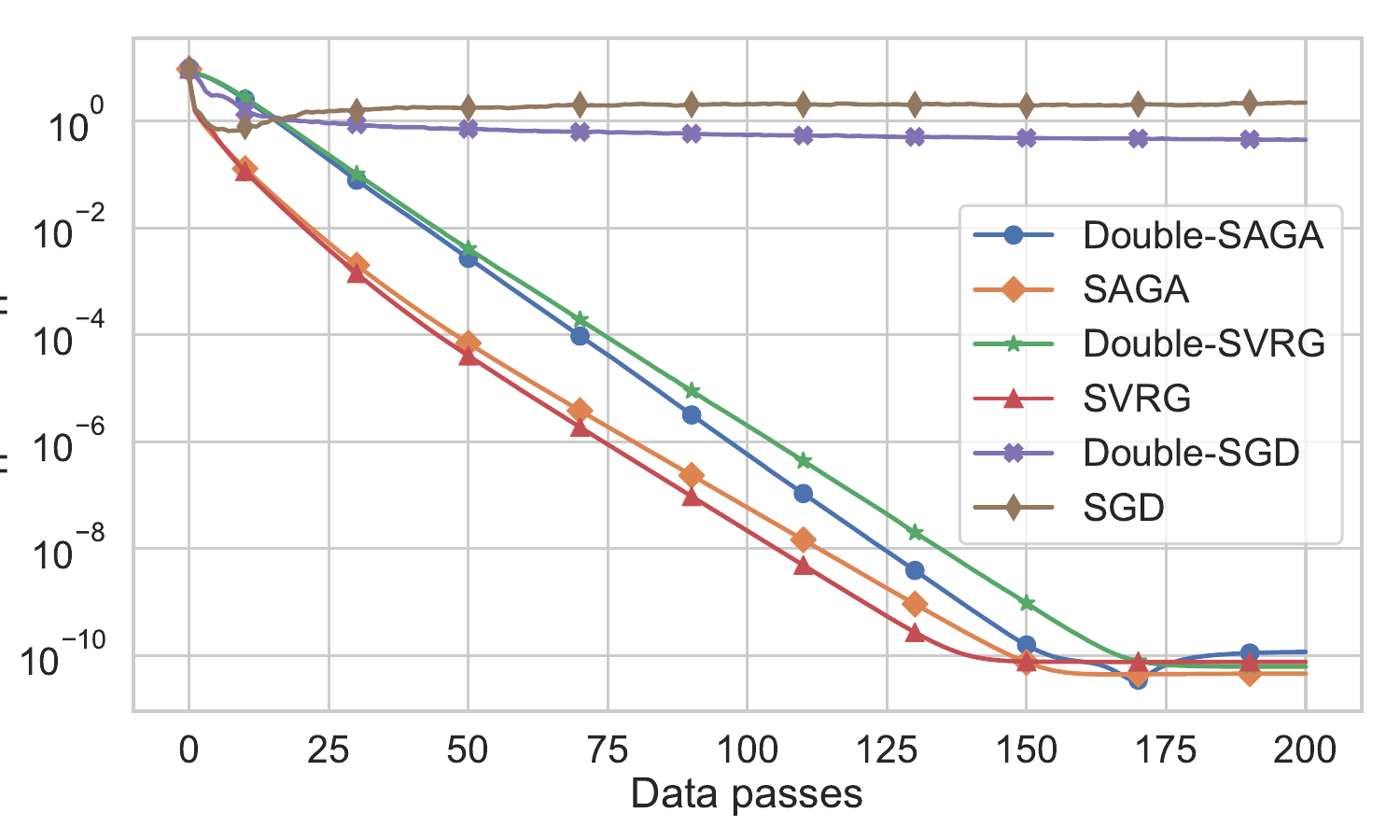}
	\caption{Left: convergence of the Stochastic Decoupling method, Kaczmarz and accelerated Kaczmarz of~\cite{gower2018accelerated} when solving $\mW x = b$ with random positive-definite $\mW\in\RR^{d\times d}$, where $d=100$. It is immediate to observe that the method we propose performs on a par with the accelerated Sketch-and-Project. Right: linear regression with A9a dataset from LIBSVM~\cite{chang2011libsvm} with first 50 observation used as linear constraints. We compare convergence of SVRG, SAGA and SGD with full projections (labeled as 'SVRG', 'SAGA', 'SGD') to the same methods combined with Algorithm~\ref{alg:sdm} (labeled as 'Double-').}
\end{figure}

\bibliographystyle{plain}
\bibliography{sdm.bib}
\clearpage
\appendix

\part*{Supplementary Material\\
\large A Stochastic Decoupling Method for Minimizing the Sum of Smooth and Non-Smooth Functions}

{\footnotesize
\tableofcontents
}

\clearpage

\section{Applications}\label{ap:applications}
In this section  we list a number of selected applications for our method:

\begin{itemize}
    \item Compressed sensing~\cite{candes2006robust}.
    \item Total Generalized Variance (TGV) for image denoising~\cite{bredies2010total}.
    \item Decentralized optimization over networks~\cite{Nedic2009distributed}.
    \item Support-vector machine~\cite{cortes1995support}.
    \item Dantzig selector~\cite{candes2007dantzig}.
    \item Group Lasso~\cite{yuan2006model}.
	\item Network utility maximization.
    \item Square-root Lasso~\cite{belloni2011square}.
    \item $\ell_1$ trend filtering~\cite{kim2009ell_1}.
    \item Convex relaxation of unsupervised image matching and object discovery~\cite{vo2019unsupervised}.
\end{itemize}

In the rest of this section we formulate some  of them explicitly. A summary of the mapping of these problems to the structure of problem \eqref{eq:pb_general} is provided in Table~\ref{tbl:summary_apps}.

\begin{table}[!h]
\caption{Selected applications of Algorithm~\ref{alg:sdm} for solving problem \eqref{eq:pb_general}.}
\centering
{
\begin{tabular}{cccc}
\toprule[.1em]
Special case of problem \eqref{eq:pb_general} & $f(x)$ & $g_j(x)$ & $R(x)$ \\
\midrule
Constrained optimization \eqref{eq:bifg9bd89d} & $f(x)$ & $\ind_{\cC_j}(x)$ & $R(x)$  \\
\addlinespace
Convex projection & $\frac{1}{2}\|x-x^0\|^2$ & $\ind_{\cC_j}(x)$ & $0$  \\
\addlinespace
Convex feasibility & $0$ & $\ind_{\cC_j}(x)$ & $0$  \\
\addlinespace
Dantzig selector \eqref{eq:Dantigg9gf98f} & $0$ & $\chi_{\cB_{\lambda}^j}(x)$ & $\|x\|_1$   \\
\addlinespace
Decentralized optimization~\eqref{eq:bid79f9hbfss} & $f_i(x_i)$ & $\ind_{\{x\;:\;w_j^\top x = 0\}}(x)$ & 0  \\
\addlinespace
Support vector machine \eqref{eq:SVMbuis} & $f(x)=\frac{\lambda}{2} \|x\|^2$, $n=1$ & $\max\{0, 1 - b_j a_j^\top x\}$ & 0   \\
\addlinespace
Overlapping group Lasso \eqref{eq:OGL} & $f_i(x)=\frac{1}{2}(a_i^\top x - b_i)^2$ & $\|x\|_{G_j}$ & 0   \\
\addlinespace
Fused Lasso  \eqref{eq:FLi0hf88fh0} & $\frac{1}{2}(a_i^\top x - b_i)^2$ & $\ind_{\cC_j^\varepsilon}(x)$ & $\lambda\|x\|_1$  \\
\addlinespace
Fused Lasso \eqref{eq:bifub98f0fss} & $\frac{1}{2}(a_i^\top x - b_i)^2$ & $\lambda_2 |\mD_{j:} x|$ & $\lambda_1\|x\|_1$  \\
\bottomrule[.1em]
\end{tabular}
} \label{tbl:summary_apps}
\end{table}

\subsection{Constrained optimization} 
Let  $\cC_j\subseteq \RR^d$ be closed convex sets with a non-empty intersection and consider the constrained composite optimization problem
\begin{align*}
    \min & \quad f(x) + R(x)
    \qquad \text{subject to} \quad \quad x \in \cap_{j=1}^m \cC_j.
\end{align*}
If we let $g_j\equiv \ind_{\cC_j}$ be the characteristic function of $\cC_j$, defined as follows: $\ind_{\cC_j}(x)=0$ for $x\in \cC_j$ and $\ind_{\cC_j}(x) =+\infty$ for $x\notin \cC_j$, this problem can be written in the form 
\begin{equation}\label{eq:bifg9bd89d}
\min_{x\in \RR^d} f(x) +R(x) + \frac{1}{m}\sum_{j=1}^m \underbrace{\ind_{\cC_j}(x)}_{g_j(x)}.
\end{equation}


For $f(x)=\frac{1}{2}\|x-x^0\|^2$ and $R\equiv 0$, this specialized to the {\em  best approximation}  problem. For $f\equiv 0$  and $R\equiv 0$, this problem specializes to the {\em convex feasibility} problem.

\subsection{Dantzig selector} Dantzig selector~\cite{candes2007dantzig} solves the problem of estimating sparse parameter $x$ from a linear model. Given an  input matrix $\mA\in\RR^{m\times d}$, output vector $b\in\RR^m$ and threshold parameter $\lambda\geq 0$, define \[\cB_\lambda\eqdef \{x\;:\;\|\mA^\top(b - \mA x)\|_{\infty}\le \lambda \} = \bigcap_{j=1}^m \cB_{\lambda}^j,\] 
where $\cB^j_\lambda \eqdef \bigl\{x\;:\; \bigl| \left(\mA^\top(b - \mA x)\right)_j \bigr| \leq \lambda\bigr\}$. The goal of the Dantzig selector problem is to find the solution to
\begin{align*}
	\min_{x\in \RR^d} \|x\|_1 + \ind_{\cB_\lambda}(x),
\end{align*}
which can equivalently be written in the  finite-sum form
\begin{equation}\label{eq:Dantigg9gf98f}
	\min_{x\in \RR^d}  \underbrace{\|x\|_1}_{R(x)} + \avejm \underbrace{\ind_{\cB_\lambda^j}(x)}_{g_j(x)}.
	\end{equation}

\subsection{Decentralized optimization} The problem of minimizing the sum of functions over a network~\cite{Nedic2009distributed} can be reformulated as
\begin{align*}
	\min_{x=(x_1, \dotsc, x_n)} \frac{1}{n}\sumin f_i(x_i) + \ind_{\{x\;:\; \mW x = 0\}}(x),
\end{align*}
where $\mW$ is a matrix such that $\mW x =0$ if and only if $x_1=\dotsb=x_n$. Functions $f_1,\dotsc, f_n$ are stored on different nodes and each node has access only to its own function. Matrix $\mW$ is often derived from a communication graph, which defines how the nodes can communicate with each other. Formally, if $\mW = (w_1^\top, \dotsc, w_m^\top)^\top$, we rewrite the problem above as
\begin{equation}\label{eq:bid79f9hbfss}
	\min_{x=(x_1, \dotsc, x_n)} \frac{1}{n}\sumin \underbrace{f_i(x_i)}_{f_i(x)} + \avejm \underbrace{\ind_{\{x\;:\;w_j^\top x = 0\}}(x)}_{g_j(x)}.
\end{equation}

\subsection{Support-vector machine (SVM)} Support-vector machine~\cite{cortes1995support} is a very popular method for supervised classification. The primal formulation of SVM is given by
\begin{equation}\label{eq:SVMbuis}
	\min_{x\in \RR^d} \underbrace{\frac{\lambda}{2} \|x\|^2}_{f(x) } + \avejm \underbrace{\max\{0, 1 - b_j a_j^\top x\}}_{g_j(x)},
\end{equation}
where $a_1,\dotsc, a_m\in \RR^d$ and $b_1, \dotsc, b_m$ are the features and the outputs. It is easy to verify that for $g_j(x) = \max\{0, 1 - b_j a_j^\top x\}$ the proximal operator is given by
\begin{align*}
	\proxj(x) = x + \Pi_{[0, \eta_j]}\biggl(\frac{1 - b_j a_j^\top x}{\|a_j\|^2} \biggr) b_j a_j.
\end{align*}

The celebrated stochastic subgradient descent method Pegasos \cite{Pegasos, Pegasos-MAPR, Pegasos2} for SVMs achieves only slow $\cO(\nicefrac{1}{t})$ rate.

\subsection{Overlapping group Lasso} 
This is a generalization of Lasso proposed in~\cite{yuan2006model} to efficiently select groups of features that are most valuable for the given objective. Let us assume that we are given sets of indices $G_1, \dotsc, G_m \subseteq	\{1,\dotsc, d\}$ and let $\|x\|_{G_j} \eqdef \sqrt{\sum_{i\in G} [x]_i^2}$, where $[x]_i$ is the $i$-th coordinate of vector $x$. Then, assuming that we are given vectors $a_1, \dotsc, a_n\in \RR^d$ and scalars $b_1, \dotsc, b_n$, the objective we want to minimize is 
\begin{equation}\label{eq:OGL}
	\min_{x\in \RR^d}  \frac{1}{n}\sumin \underbrace{\frac{1}{2}(a_i^\top x - b_i)^2}_{f_j(x)} + \avejm \underbrace{\|x\|_{G_j}}_{g_j(x)}.
\end{equation}
It is easy to verify that if $g_j(x) = \|x\|_{G_j}$, then
\begin{align*}
	[\prox_{\eta_j g_j}(x)]_i
	=\begin{cases}
	[x]_i, & \text{if } i\not\in G_j, \\
	\max\left\{0, \left(1 - \frac{\eta_j}{\|x\|_{G_j}}\right) \right\}[x]_i, & \text{if } i\in G_j.
	 \end{cases}
\end{align*}
Vector $y_j^t$ will always have at most $|G_j|$ nonzeros, so one can store in memory only the coordinates of $y_j^t$ from $G_j$.

\subsection{Fused Lasso} The Fused Lasso problem~\cite{tibshirani2005sparsity} is defined as
\begin{equation}\label{eq:FLi0hf88fh0}
	\min_{x\in \RR^d} \frac{1}{n}\sumin \underbrace{\frac{1}{2}(a_i^\top x - b_i)^2}_{f_i(x)} + \underbrace{\lambda\|x\|_1}_{R(x)} + \frac{1}{d-1}\sum_{j=1}^{d-1} \underbrace{\ind_{\cC_j^\varepsilon}(x)}_{g_j(x)},
\end{equation}
where $\cC_j^\varepsilon \eqdef \left\{x\;:\; \left| [x]_j - [x]_{j+1} \right| \le \varepsilon \right\},$
$[x]_j$ is the $j$-th entry of vector $x$, $a_1,\dotsc, a_n\in \RR^d$ and $b_1, \dotsc, b_n\in\RR$ are given vectors and scalars, $\varepsilon$ is given thresholding parameter.

Another formulation of the Fused Lasso is done by using penalty functions. Define $\mD$ to be zero everywhere except for $\mD_{i,i}=1$ and $\mD_{i, i+1}=-1$ with $i=1,\dotsc, d-1$. Note that $\|\mD x\|_1 = \sum_{j=1}^m |\mD_{j:} x|$, where $m$ is the number of rows of $\mD$. Then the reformulated objective is
\begin{equation}\label{eq:bifub98f0fss}
	\min_x \frac{1}{n}\sumin \underbrace{\frac{1}{2}(a_i^\top x - b_i)^2}_{f_i(x)} + \underbrace{\lambda_1\|x\|_1}_{R(x)} + \avejm \underbrace{\lambda_2 |\mD_{j:} x|}_{g_j(x)}.
\end{equation}
In our notation, this means $\mA = \mD^\top$ and $\mA^\top \mA$ is a tridiagonal matrix given by
\begin{align*}
	\mA^\top \mA
	=
	\begin{pmatrix}
		2   & -1  \\
		-1 &    2 & -1  \\
		     &  -1 & 2   & -1 & \\
		     &       &      & \ddots    & -1 \\
		     &       &      &           -1  & 2
	\end{pmatrix}.
\end{align*} 
Let $\mW$ be a tridiagonal matrix of size $(d-1)\times (d-1)$ with $a$ on its main diagonal and $b$ on the other two diagonals. It can be shown that its eigenvalues  are given by $\lambda_k(\mW) = a+2|b|\cos \left(\frac{k\pi}{d} \right)$, $k=1,\dotsc, d-1$. Thus, $\lambda_{\min}(\mA^\top \mA) = 2 + 2\cos\left(\left(1 - \frac{1}{d} \right)\pi\right) = 2 - 2\cos\left(\frac{\pi}{d}\right) \approx \frac{1}{2d^2}$ and $\min_j \frac{1}{\|\mA_j\|^2} = \frac{1}{6}$. Therefore, if in~\eqref{eq:FLi0hf88fh0} or~\eqref{eq:bifub98f0fss} $\lambda_1=0$, we guarantee linear convergence with the aforementioned constants.

\subsection{Square-root Lasso} The approach gets its name from minimizing the square root of the regular least squares, i.e., $\|\mD w - b\|$ instead of $\|\mD w - b\|^2$. This is then combined with $\ell_1$-penalty for feature selection, which gives the objective
\begin{align*}
		\min_{w\in \RR^d} \|\mD w - b\| + \lambda \|w\|_1.
\end{align*}
Equivalently, by introducing a new variable $z$ we can put constraints $\mD_{j:} x - [z]_j =0$ for $j=1,\dotsc, m$, which can be written as $a_j^\top(w^\top, z^\top)^\top=0$ with $a_j = (\mD_{j:}, e_j^\top)^\top$ and $e_j\eqdef (0, 0, \dotsc, \underbrace{1}_{j}, \dotsc, 0)$. Then, the reformulation is
\begin{align*}
	\min_{x=(w, z)\in \RR^{d+m}} \avejm \underbrace{\ind_{\{x:a_j^\top x =0\}}}_{g_j(x)=g_j(w, z)} + \underbrace{\|z - b\| + \lambda \|w\|_1}_{R(x)=R(w,z)}.
\end{align*}
The proximal operator of $R$ is that of a block-separable function, which is easy to evaluate:
\begin{align*}
	\prox_{\eta R}(x) 
	= \begin{pmatrix}\prox_{\eta\lambda\|\cdot\|_1}(w) \\ \prox_{\eta \|\cdot - b\|}(z)\end{pmatrix}.
\end{align*}

\newpage

\section{Relation to Existing Methods}

\subsection{SDCA, Dykstra's algorithm and the Kaczmarz method}
Here we formulate SDCA~\cite{shalev2013stochastic}, Dykstra's algorithm and Kaczmarz method. SDCA is a method for solving
\begin{align*}
	\min_{x\in\RR^d} \avejm g_j(x)+ \frac{1}{2}\|x - x^0\|^2.
\end{align*}
If $j$ is sampled uniformly from $\{1,\dotsc, m\}$, SDCA iterates can be defined by the following recursion,
\begin{align*}
	x^{t+1} 
	&= \prox_{\eta g_j}(x^t + \overline y_j^t), \\
	\overline y_j^{t+1}
	&= \overline y_j^t + x^t - x^{t+1},
\end{align*}
If we restrict our attention to characteristic functions, i.e., 
\begin{align*}
	g_j(x) = \ind_{\cC_j}(x) = \begin{cases}0, &\text{if } x\in \cC_j \\ +\infty, & \text{otherwise} \end{cases},
\end{align*}
then the proximal operator step is replaced with projection:
\begin{align*}
	x^{t+1} 
	&= \Pi_{\cC_j}(x^t + \overline y_j^t).
\end{align*}
This is known as Dykstra's algorithm. Finally, if $\cC_j= \{x: a_j^\top x=b_j\}$, then it boils down to random projections, i.e.,
\begin{align*}
	x^{t+1} = \Pi_{\{a_j^\top x = b_j\}} (x^t),
\end{align*}
which is the method of Kaczmarz.

\begin{theorem}\label{th:kaczmarz}
	Consider the regularized minimization problem of SDCA, which is 
	\begin{align*}
		\min_x \avejm g_j(x)	 + \frac{1}{2}\|x - x^0\|^2
	\end{align*}		
	with convex $g_1, \dotsc, g_m$. Then, SDCA is a special cases of Algorithm~\ref{alg:sdm} obtained by applying it with $f(x)=\frac{1}{2}\|x - x^0\|^2$, $R(x)\equiv 0$, stepsize $\eta= \frac{1}{m}$ and initialization $y_1^0=\dotsb=y_m^0 =0$. Furthermore, if we consider special case $g_j = \ind_{\cC_j}$, where $\cC_j\neq \emptyset$ is a closed convex set, then we also obtain Dykstra's algorithm, and if every $\cC_j$ is a linear subspace, then we recover the Kaczmarz method.
\end{theorem}
\begin{proof}
	Consider the iterates of SDM. We will show by induction that $y^t = x^0-x^t$ and $x^{t+1} = \proxj(x^t + \eta_j y_j^t)$.
	Indeed, it holds for $y^0$ by initialization, and then by induction assumption we have
	\begin{align*}
		z^t
		= x^t - \eta(x^t - x^0) - \eta y^t
		= x^t - \eta(x^t - x^0) - \eta (x^0 - x^t)
		= x^t.
	\end{align*}
	Therefore, if we denote $\overline y_j^t \eqdef \eta_j y_j^t$, then
	\begin{align*}
		x^{t+1} 
		=\proxj(x^t + \overline y_j^t),
	\end{align*}
	which is the update rule of $x^{t+1}$ in SDCA. Moreover, we have
	\begin{align*}
		\overline y_j^{t+1} 
		= \eta y_j^{t+1}
		=\eta y_j^t + x^t - x^{t+1}
		= \overline y_j^t + x^t - x^{t+1}.
	\end{align*}
	Finally, by induction assumption it holds $y^t=x^0 - x^t$, whence
	\begin{align*}
		y^{t+1}
		= y^t + \frac{1}{m}(y_j^{t+1} - y_j^{t})
		= y^t + z^t - x^{t+1}
		= x^0 - x^t + x^t - x^{t+1}
		= x^0 - x^{t+1},
	\end{align*}
	which yields our induction step and the proof itself.
\end{proof}
\subsection{Accelerated Kaczmarz}\label{ap:kaczmarz}
Accelerated Kaczmarz~\cite{liu2016accelerated} performs the following updates:
\begin{align*}
	z^t
	&= (1 - \alpha_t) x^t - \alpha_t y^t,\\
	x^{t+1}
	&= \Pi_{\{x: a_i^\top x = b_i\}} (z^t), \\
	y^{t+1}
	&= y^t + \gamma_t (z^t - x^{t+1}) + (1 - \beta_t) (z^t - y^t)
\end{align*}
with some parameters $\alpha_t, \gamma_t, \beta_t$. While the original analysis~\cite{liu2016accelerated} suggests $\beta_t<1$, our method gives the same update when $f(x)=\frac{1}{2}\|x\|^2$, $R\equiv 0$, $\alpha_t=\eta$, $\beta_t=1$, $\gamma_t=\frac{1}{\eta n}$. 
\subsection{ADMM and Douglas-Rachford splitting}
ADMM, also known as Douglas-Rachford splitting, in its simplest form as presented in~\cite{parikh2014proximal} is a special case of Algorithm~\ref{alg:sdm} when $f\equiv 0$ and $m=1$.
\subsection{Point-SAGA, SAGA, SVRG and Proximal GD}
In the trivial case $f\equiv 0$ and $R\equiv 0$, we recover Point-SAGA. Methods such as SAGA, SVRG and Proximal Gradient Descent are obtained, in contrast, by setting $g\equiv 0$. We would like to mention that introducing $g$ does not change the stepsizes for which those methods work, e.g., Gradient Descent works with arbitrary $\eta < \nicefrac{2}{L}$, which is tight. The similarity suggests that small $\eta$ should be used when solving this problem and this observation is validated by our experiments.

\subsection{Stochastic Primal-Dual Hybrid Gradient}\label{ap:spdhg}
The relation to the Stochastic Primal-Dual Hybrid Gradient (SPDHG) is complicated. On the one hand, SPDHG is a general method with three parameters and it preconditions proximal operators with matrices, so our method cannot be its strict generalization. On the other hand, SPDHG does not allow for $f$. Moreover, when $f\equiv 0$ and some parameters are set to specific values in SPDHG, the methods coincide, but the guarantees are not the same. In particular, we show below that one of the parameters in SPDHG, $\theta$, should be set to 1, in which case linear convergence for smooth $g_1, \dotsc, g_m$ was not known for SPDHG. Therefore, the tools developed in this work can potentially lead to new discoveries about full version of SPDHG as well.

Let us now formulate the method explicitly. After a simple rescaling of the functions, SPDHG from~\cite{ehrhardt2017faster} can be formulated as a method to solve the problem
\begin{align}
	\min_{x\in\RR^d} \avejm \phi_j(\mA_j^\top x) + R(x).\label{pb:spdhg}
\end{align}
Renaming the variables for our convenience and choosing for simplicity uniform probabilities of sampling $j$ from $\{1,\dotsc, m\}$, the update rules of SPDHG can be written as
\begin{align*}
	w^t 
	&= \proxR(w^{t-1} - \eta \overline y^t), \\
	y_j^{t+1} 
	&= \prox_{\sigma \phi_j^*}(\sigma\mA_j^\top w^t + y_j^t), \\
	y^{t+1} 
	&= y^t + \frac{1}{m}\mA_j(y_j^{t+1} - y_j^t), \\
	\overline y^{t+1} 
	&= y^t + \theta \mA_j(y_j^{t+1} - y_j^t),
\end{align*}
where $\eta, \sigma$ and $\theta$ are the method's parameters and $\phi_j^*$ is the Fenchel conjugate of $\phi_j$. The initialization that we are interested in is with $y^0 = \avejm y_j^0$, $\overline y^0 = y^0$, $w^0 = x^0$.

One can immediately see that one big difference with our approach is that the method puts $\mA_j$ outside of the proximal operator, which also leads to different iteration complexity. In particular, when $\phi_1,\dotsc, \phi_m$ are smooth, the complexity proved in~\cite{chambolle2018stochastic} is 
\begin{align*}
	\cO\left(\left(m + \sumjm \|\mA_j\|\sqrt{\frac{L_\phi}{\mu_R}} \right)\log \frac{1}{\varepsilon}\right),
\end{align*}
where $\mu_R$ is the strong convexity constant of $R$ and $L_\phi$ is the smoothness constant of $\phi_1,\dotsc, \phi_m$. Since function $g_j(x) = \phi_j(\mA_j^\top x)$ is at most $L_\phi \|\mA_j\|^2$ smooth, our rate from Corollary~\ref{cor:acc_in_g} with $\mu$-strongly convex and $L$-smooth $f$ is 
\begin{align*}
	\cO\left(\left(n+m + \frac{L}{\mu} +  \sqrt{m\frac{L_\phi}{\mu}}\max_j \|\mA_j\| \right)\log \frac{1}{\varepsilon}\right).
\end{align*}
If, in addition, we use sampling with probabilities proportional to $\|\mA_j\|$, then we can achieve
\begin{align*}
	\cO\left(\left(n+m + \frac{L}{\mu} +  \frac{1}{\sqrt{m}}\sumjm \|\mA_j\|\sqrt{\frac{L_\phi}{\mu}} \right)\log \frac{1}{\varepsilon}\right).
\end{align*}

We do not prove this, but the complexity for our method will be similar if we use strongly convex $R$ rather than $f$, so our rates should match or be even be superior to that of SPDHG, at the cost of evaluating potentially harder proximal operators.

Now, let us prove that our method is indeed connected to SPDHG via choice of $\theta=1$ and $\eta\sigma = 1$.
\begin{theorem}
	If we apply SPDHG with identity matrices $\mA_j = \mI$, i.e., $\phi_j(x) = g_j(x)$, and choose parameters $\theta=1$ and $\eta\sigma = 1$, then it is algorithmically equivalent to Algorithm~\ref{alg:sdm} with $f\equiv 0$.
\end{theorem}
\begin{proof}
	Since $\phi_j$ and $g_j$ are the same, we will use in the proof $g_j$ only. 
	
	First, mention that it is straightforward to show by induction that $y^t = \avejm y_j^t$, which coincides with our update. Our goal is to show by induction that in SPDHG it holds
	\begin{align*}
		w^{t-1} - \eta \overline y^t
		= x^t - \eta y^t,
	\end{align*}
	where we define sequence $x^t$ as 
	\begin{align*}
		x^{t+1} 
		\eqdef \prox_{\eta g_j}(w^t + \eta y_j^t)
		= \prox_{\frac{1}{\sigma} g_j}(w^t + \eta y_j^t).
	\end{align*}
	We will see that implicitly $x^{t+1}$ is present in every update of SPDHG. To this end, let us first rewrite the update for $y_j^{t+1}$. We have by Moreau's identity
	\begin{align*}
		y_j^{t+1}
		= \prox_{\sigma g_j^*}(\sigma w^t + y_j^t)
		= \sigma w^t + y_j^t - \sigma\prox_{\frac{1}{\sigma} g_j}\left(\frac{\sigma w^t + y_j^t}{\sigma} \right).
	\end{align*}
	Since we consider $\sigma= \frac{1}{\eta}$, it transforms into
	\begin{align*}
		y_j^{t+1}
		= y_j^t  + \frac{1}{\eta}\left( w^t - \prox_{\eta g_j}(w^t + \eta y_j^t) \right) 
		&= y_j^t  + \frac{1}{\eta}\left( w^t - x^{t+1} \right)
	\end{align*}
	The only missing thing is rewriting update for $w^t$ in terms of $x^t$ and $y^t$. From the update rule for $y_j^{t+1}$ we derive
	\begin{align*}
		\overline y^{t+1}
		= y^t + \theta (y_j^{t+1} - y_j^t)
		= y^t + \frac{\theta }{\eta}(w^t - x^{t+1}).
	\end{align*}
	Hence,
	\begin{align*}
		w^{t+1} 
		= \proxR(w^t - \eta \overline y^{t+1})
		= \proxR(w^t - \eta  y^{t+1} - \theta (w^t - x^{t+1}))
		\overset{\theta=1}{=} \proxR(x^{t+1} - \eta  y^{t+1}).
	\end{align*}
	Thus, updates for $w^t$, $y_j^t$ and $y^t$ completely coincide under this choice of parameters.
\end{proof}
Since our method under $f\equiv 0$ reduces to Point-SAGA, we obtain the following result that was unknown.
\begin{corollary}
	Point-SAGA~\cite{defazio2016simple} is a special case of Stochastic Primal-Dual Hybrid Gradient~\cite{chambolle2018stochastic}.
\end{corollary}

\newpage

\section{Evaluating Proximal Operators}
For some functions, the proximal operator admits a closed form solution, for instance if $g_j( x) = \ind_{\{x\;:\; a_j^\top x=b_j\}}(x)$, then
\begin{align*}
	\proxj(x)
	= x - \frac{a_j^\top x - b_j}{\|a_j\|^2}a_j.
\end{align*}
If, however, the proximal operator is not given in a closed form, then it is still possible to efficiently evaluate it. If $g_j=\phi_j(\mA_j^\top x)$, $\mA_j\in \RR^{d\times d_j}$, then the proximal operator is the solution of a $d_j$-dimensional strongly convex problem.
\begin{lemma}\label{lem:prox_of_composition}
	Let $\phi_j \colon \RR^{d_j}\to \RR$ be a convex lower semi-continuous function such that $\Range{\mA_j^\top}$ has a point of $\dom \phi$. If $g_j(x) = \phi_j(\mA_j^\top x)$, then
	\begin{align*}
		x - \prox_{\eta_j g_j}(x)
		\in \Range{\mA_j}.
	\end{align*}
\end{lemma}
\begin{proof}
	Let us fix $x$. Any vector $z\in\RR^d$ can be decomposed as $z = x + \mA_j\beta + w$, where $\beta\in\RR^{d_j}$ and $\mA_j^\top w = 0$, from which it also follows $g_j(z) = \phi_j(\mA_j^\top x+ \mA_j^\top \mA_j \beta)$. Then
	\begin{align*}
		\prox_{\eta_j g_j}(x)
		&\eqdef \argmin_{z\in\RR^d}\left\{\eta_j \phi_j(\mA_j^\top z) + \frac{1}{2}\|z - x\|^2 \right\}\\
		&= \argmin_{z=x+\mA_j\beta + w}\left\{\eta_j \phi_j(\mA_j^\top x + \mA_j^\top \mA_j \beta) + \frac{1}{2}\|\mA_j\beta + w\|^2 \right\} \\
		&= \argmin_{z=x+\mA_j\beta + w}\left\{\eta_j \phi_j(\mA_j^\top x + \mA_j^\top \mA_j \beta) + \frac{1}{2}\|\mA_j\beta\|^2 + \frac{1}{2}\| w\|^2 \right\} .
	\end{align*}
	Clearly, the last expression achieves its minimum only when $w=0$.
\end{proof}

We can simplify the expression for the proximal operator even further if $\mA_j$ is of full column rank, for instance if it is just a single nonzero row. It is straightforward to verify that for any matrix $\mB\in\RR^{d_1\times d_2}$, constant vector $c\in\RR^{d_2}$ and function $\Phi$ with a unique minimizer and $\dom \Phi(\mB \beta +c) \neq \emptyset$ it holds
\begin{align}
	\argmin_{\beta=\mB(\alpha + c), \beta\in\RR^{d_2}}\Phi(\beta)
	&= \argmin_{\beta=\mB(\alpha + c), \beta\in\RR^{d_2}} \Phi(\mB(\alpha + c)) \notag \\
	&= \mB\argmin_{u=\alpha + c, \alpha\in\RR^{d_1}} \Phi(\alpha + c) \notag \\
	&= \mB\Bigl(\argmin_{u\in\RR^{d_1}} \Phi( u ) - c\Bigr). \label{eq:argmin_rule}
\end{align}
Since we know by Lemma~\ref{lem:prox_of_composition} that $u\eqdef \proxj(x) = x + \mA_j \beta_j$ for some $\beta_j\in\RR^{d_j}$, we can write the necessary and sufficient optimality condition for $u$ by repeatedly applying~\eqref{eq:argmin_rule}
\begin{align*}
	\proxj(x) 
	&= \argmin_{u = x + \mA_j\beta,\; \beta\in\RR^{d_j}} \left\{\phi_j(\mA_j^\top u) + \frac{1}{2\eta_j}\|x - u\|^2 \right\}\\
	&\overset{\eqref{eq:argmin_rule}}{=} x + \mA_j \argmin_{\beta\in\RR^{d_j}} \left\{\phi_j\left(\mA_j^\top (x + \mA_j\beta)\right) + \frac{1}{2\eta_j} \|\mA_j\beta\|^2\right\} \\
	&=x + \mA_j \argmin_{\beta\in\RR^{d_j}} \left\{\phi_j\left(\mA_j^\top x + \mA_j^\top\mA_j\beta\right) + \frac{1}{2\eta_j} \|\mA_j(\mA_j^\top\mA_j)^{-1}\mA_j^\top\mA_j\beta\|^2\right\} \\
	&\overset{\eqref{eq:argmin_rule}}{=}x + \mA_j (\mA_j^\top\mA_j)^{-1}\argmin_{\alpha=\mA_j^\top\mA_j\beta} \left\{\phi_j\left(\mA_j^\top x + \alpha\right) + \frac{1}{2\eta_j} \|\mA_j(\mA_j^\top\mA_j)^{-1}\alpha\|^2\right\} \\
	&\overset{\eqref{eq:argmin_rule}}{=}x + \mA_j (\mA_j^\top\mA_j)^{-1}\Bigl(\argmin_{\theta=\alpha + \mA_j^\top x} \left\{\phi_j\left(\theta\right) + \frac{1}{2\eta_j} \|\mA_j(\mA_j^\top\mA_j)^{-1}(\theta - \mA_j^\top x)\|^2\right\} - \mA_j^\top x\Bigr).
\end{align*}
Note that
\begin{align*}
	\|\mA_j(\mA_j^\top\mA_j)^{-1}(\theta - \mA_j^\top x)\|^2 
	&= (\theta - \mA_j^\top x)^\top (\mA_j^\top\mA_j)^{-1}\mA_j^\top\mA_j(\mA_j^\top\mA_j)^{-1}(\theta - \mA_j^\top x) \\
	&= \|\theta - \mA_j^\top x\|^2_{(\mA_j^\top\mA_j)^{-1}},
\end{align*}
where for any positive semi-definite matrix $\mW$ we denote $\|x\|_{\mW}^2\eqdef x^\top \mW x$. Denoting similarly $\prox^{\mW}_{\eta_j \phi_j}(x)\eqdef \argmin_{\theta}\{\phi_j(\theta) + \frac{1}{2\eta_j}\|\theta - x\|_\mW^2\}$, we obtain
\begin{align*}
	\proxj(x) 
	&=x + \mA_j (\mA_j^\top\mA_j)^{-1}\left( \argmin_{\theta\in\RR^{d_j}} \left\{\phi_j\left(\beta\right) + \frac{1}{2\eta_j} \|\theta - \mA_j^\top x\|_{(\mA_j^\top\mA_j)^{-1}}\right\}  - \mA_j^\top x \right) \\
	&=x + \mA_j (\mA_j^\top\mA_j)^{-1}\left( \prox_{\eta_j\phi_j}^{(\mA_j^\top\mA_j)^{-1}}\left(\mA_j^\top x\right)  - \mA_j^\top x \right).
\end{align*}
Thus, we only need to know how to efficiently evaluate $\prox_{\lambda \phi_j}^{(\mA_j^\top\mA_j)^{-1}}(z)$ for arbitrary $\lambda>0$ and $z\in\RR^{d_j}$, assuming that matrix $(\mA_j^\top\mA_j)^{-1}$ can be precomputed. For example, if $\mA_j=a_j\in\RR^d$, then
\begin{align*}
	\prox_{\eta_j \phi_j}^{(a_j^\top a_j)^{-1}}(x) = \prox_{\eta_j \|a_j\|^2 \phi_j}(x).
\end{align*}

If, in addition, $\phi_j\colon \RR \to \RR$ is given by
\begin{align*}
	\phi_j(z) = \begin{cases} b_jz, & \text{if } z\le 0,\\ c_j z, & \text{otherwise} \end{cases}
\end{align*}
with some $b_j, c_j\in\RR$, $b_j<c_j$, then $\prox_{\lambda \phi_j}(z) = z - \lambda b_j $ for $z\le\lambda b_j$, $\prox_{\lambda \phi_j}(z)=0$ for $z\in(\lambda b_j, \lambda c_j]$ and $\prox_{\lambda \phi_j}(z)=z - \lambda c_j$ for $z> \lambda c_j$. Therefore,
\begin{align*}
	\prox_{\eta_j g_j}(x)
	= \begin{cases} x - \eta_j a_j b_j, &\text{if } a_j^\top x\le \|a_j\|^2 b_j,\\ x - \frac{a_j^\top x}{\|a_j\|^2}a_j, & \text{if } \|a_j\|^2b_j\le a_j^\top x \le \|a_j\|^2c_j , \\ x - \eta_j a_j c_j, &\text{otherwise} \end{cases}.
\end{align*}
Note that if $\|a_j\|^2b_j\le a_j^\top x \le \|a_j\|^2c_j $, then $a_j^\top \prox_{\eta_j g_j}(x) =0 $.

\newpage
\section{Inequalities Related to Smoothness, Convexity and Proximal Operators} \label{sec:basic_facts}

Since many of our proofs are easier to write when one uses Bregman divergences, we will formulate most of the required properties in terms of $D_f(\cdot, \cdot)$.

\begin{proposition}
	Let $f$ be convex and $L$-smooth, then we have for any $x, y$
	\begin{align}
		&\|\nabla f(x) - \nabla f(y)\|^2
		\le 2L D_f(x, y), \label{eq:grad_dif_bregman}\\
		&\|\nabla f(x) - \nabla f(y)\|^2
		\le L\<\nabla f(x) - \nabla f(y), x - y>.\label{eq:grad_dif_scalar_prod}
	\end{align}
\end{proposition}
\begin{proposition}
	Let $f$ be $\mu$-strongly convex, including the case $\mu=0$, which holds when $f$ is simply convex. Then, for arbitrary $x$ and $y$
	\begin{align}
		\frac{\mu}{2}\|x - y\|^2 + D_f(x, y)
		\le \<\nabla f(x) - \nabla f(y), x - y>.\label{eq:scal_prod_cvx}
	\end{align}
\end{proposition}
The proposition above is convenient for proofs of SVRG and SAGA, but it is not tight if we want to show that Gradient Descent converges for any $\eta\le \frac{2}{L + \mu}$ when the objective is $\mu$-strongly convex. To make the analysis tighter, we require the following statement.
\begin{proposition}
	Let $f$ be differentiable and $\mu$-strongly convex. Then we have for any $x$ and $y$
	\begin{align}
		\mu\| x - y\|^2
		\le \<\nabla f(x) - \nabla f(y), x - y>. \label{eq:scal_prod_str_cvx}
	\end{align}
	Moreover, if $f$ is also $L$-smooth, then
	\begin{align}
		\frac{\mu L}{L + \mu}\|x - y\|^2 + \frac{1}{L + \mu}\|\nabla f(x) - \nabla f(y)\|^2
		\le \<\nabla f(x) - \nabla f(y), x - y>. \label{eq:scal_prod_tight_str_cvx}
	\end{align}
\end{proposition}
This is the tightest inequality one can get and, in particular, \eqref{eq:scal_prod_tight_str_cvx} implies \eqref{eq:grad_dif_scalar_prod} when $\mu=0$.

\begin{figure}[!h]
	\center
	\includegraphics[scale=0.04]{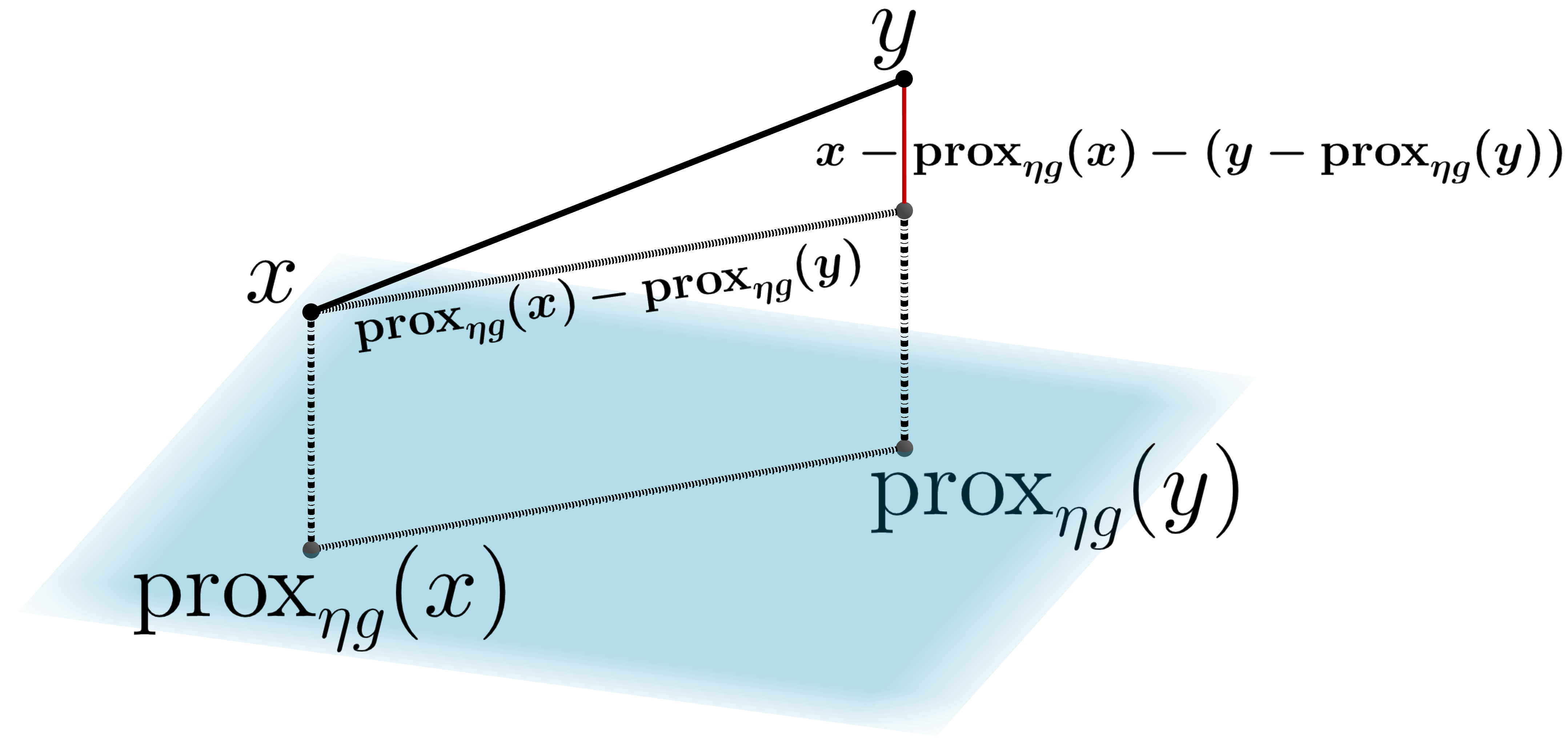}
	\caption{Illustration of property~\eqref{eq:nonexp} with characteristic function of a linear subspace, $g(x)=\ind_{\{x \;:\; a^\top x = b\}}$. In this case the proximal operator returns the projection of a point onto the subspace, and Inequality~\eqref{eq:nonexp} becomes identity and follows from Pythagorean theorem.}
	\label{fig:prox}
\end{figure}

An important property of the proximal operator is firm non-expansiveness:
\begin{proposition} Let $g:\RR^d \to \RR \cup \{+\infty\}$ be a proper closed  convex function. Then its proximal operator is firmly non-expansive. That is,  for all $\eta\in \RR$, 
\begin{align}
    \|\prox_{\eta g}(x) - \prox_{\eta g}(z)\|^2
    \le \|x - z\|^2 - \left(1 + \frac{1}{\eta L_g}\right)\|x - \prox_{\eta g}(x) - (z - \prox_{\eta g}(z))\|^2, \label{eq:nonexp}
\end{align}
where $L_g\in\RR\cup \{+\infty\}$ is the smoothness constant of function $g$ (for non-smooth functions, $L_g=+\infty$).
\end{proposition}

Inequality~\eqref{eq:nonexp} was also the main inspiration for our proofs. We derived the method by playing with this inequality and trying to see how it can be combined with a full gradient step $\nabla f$, and later extended it. Moreover, we would like to note that Equation~\ref{eq:nonexp} is tight if $g(x) = \ind_{\{x\;:\;a^\top x = b\}}$ for some vector $a$ and scalar $b$, as is shown in Figure~\ref{fig:prox}.

\clearpage
\section{Optimality Conditions} \label{sec:Opt_Cond}

We now comment on the nature of Assumption~\ref{as:optimality}.
In view of the first-order necessary and sufficient condition for the solution of~\eqref{eq:pb_general}, we have
\begin{align*}
    x^* \in \cX^*  \quad \Leftrightarrow \quad  0 \in \partial F(x^*) = \nabla f(x^*) + \partial (g + R)(x^*).
\end{align*}
By the weak sum rule \cite[Cor 3.38]{beck-book-first-order}, we have \[\partial F(x)  \supseteq  \nabla f(x) + \frac{1}{m}\sum_{j=1}^m\partial g_j(x) +\partial  R(x)\] for all $x\in \dom F \supseteq \cX^*$. Under the regularity condition $\cap_{j=1}^k(\dom g_j) \cap_{j=k+1}^m {\rm ri} (\dom g_j)\cap {\rm ri}( \dom R) \neq \emptyset$, where $g_1,\dotsc, g_k$ are polyhedral functions, the inclusions becomes an identity \cite[Thm 23.8]{rockafellar2015convex}, which means that Assumption~\ref{as:optimality} is satisfied.

For functions $g_j$ of the form $g_j(x) = \phi_j(\mA_j^\top x)$, where $\phi_j:\RR^{d_j}\to \RR\cup \{+\infty\}$ are proper closed convex functions and $\mA_j\in \RR^{d\times d_j}$, we shall instead consider the following (slightly stronger) assumption:
\begin{assumption}\label{as:optimality2}   There exists $x^*\in \cX^*$ and vectors $y_1^*\in \mA_1 \partial \phi_1( \mA_1^\top x^*), \dots, y_m^*\in  \mA_m \partial \phi_m( \mA_m^\top x^*)$ and $r^*\in \partial R(x^*)$ such that $        \nabla f(x^*) + \avejm y_j^* + r^* = 0.$
\end{assumption}

Since $\mA_j \partial \phi_j( \mA_j^\top x) \subseteq \partial g_j(x)$ for all $x\in \dom g_j$  \cite[Thm 3.43]{beck-book-first-order}, Assumption~\ref{as:optimality2}  is indeed stronger than Assumption~\ref{as:optimality}. 
If $\Range{\mA_j^\top}$ contains a point from ${\rm ri}(\dom g_j)$, or $g_j$ is polyhedral and $\Range{\mA_j^\top}$ contains a point from mere $\dom g_j$, then $\partial g_j(x) = \mA_j\partial \phi_j(\mA_j^\top x)$ for any $x$ \cite[Thm 23.9]{rockafellar2015convex}, and these two assumptions are the same.

Below we provide another stationarity condition that shows why $x^*$ is a fixed-point of our method.

\begin{lemma}[Optimality conditions]\label{lem:optimality}
    Let $x^*$ be a solution of~\eqref{eq:pb_general} and let Assumption~\ref{as:optimality} be satisfied. Then for any $\eta,\eta_j\in \RR$,
    \begin{align*}
        x^* &= \proxR(x^* - \eta\nabla f(x^*) - \eta y^*), \qquad x^* = \proxj(x^* + \eta_j y_j^*).
    \end{align*}    
\end{lemma}
\begin{proof}
    Let $z = \proxR(x^* - \eta\nabla f(x^*) - \eta y^*) = \argmin_u \{\eta R(u) + \frac{1}{2}\|u - (x^* - \eta\nabla f(x^*) - \eta y^*)\|^2 \}$. $R$ is convex, so the problem inside $\argmin$ is strongly convex, and the necessary and sufficient condition for $z$ to be its solution is
    \begin{align*}
        0\in z - x^* + \eta\nabla f(x^*) + \eta y^* + \eta \partial R(z).
    \end{align*}
    By Assumption~\ref{as:optimality} it holds for $z=x^*$, implying the first equation that we want to prove. The second one follows by exactly the same argument applied to $\argmin_u \{\eta_j g_j(u) + \frac{1}{2}\|u - (x^* + \eta_j y_j^*)\|^2 \}$.
\end{proof}


\clearpage
\section{Convergence Proofs}
In this section, we provide the proofs of our convergence results. Each lemma, theorem and corollary is first restated and only then proved to simplify the reading.

\subsection{Proof of Lemma~\ref{lem:gd} (Gradient Descent)}\label{ap:gd}


Here we prove that Gradient Descent update on $f$ satisfies our assumption on the method with the best possible stepsizes.
\begin{customlem}{\ref{lem:gd}}
	If $f$ is convex, Gradient Descent satisfies Assumption~\ref{as:method}(a) with any $\eta_0 < \frac{2}{L}$, $\omega = 2 - \eta_0 L$ and $\cM^t = 0$. If $f$ is $\mu$-strongly convex, Gradient Descent satisfies Assumption~\ref{as:method}(b) with $\eta_0 = \frac{2}{L + \mu}$, $\omega = 1$ and $\cM^t=0$.
\end{customlem}
\begin{proof}
	Since we consider Gradient Descent, we have
	\begin{align*}
		w^t
		&= x^t - \eta \nabla f(x^t).
	\end{align*}
	First, if $f$ is convex and smooth, then for any $\eta \le \eta_0 < \frac{2}{L}$
	\begin{align*}
		\|w^t - w^*\|^2
		&= \|x^t - x^*\|^2 - 2 \eta\<\nabla f(x^t) - \nabla f(x^*), x^t - x^*> + \eta^2\|\nabla f(x^t) - \nabla f(x^*)\|^2 \\
		&\le \|x^t - x^*\|^2 - \eta(2 - \eta_0 L)\<\nabla f(x^t) - \nabla f(x^*), x^t - x^*> \\
		&\quad -  \eta \eta_0 L\<\nabla f(x^t) - \nabla f(x^*), x^t - x^*> + \eta\eta_0\|\nabla f(x^t) - \nabla f(x^*)\|^2 \\
		&\overset{\eqref{eq:grad_dif_scalar_prod}}{\le} \|x^t - x^*\|^2 - \eta(2 - \eta_0 L)\<\nabla f(x^t) - \nabla f(x^*), x^t - x^*>\\
		 &\overset{\eqref{eq:scal_prod_cvx}}{\le} \|x^t - x^*\|^2 - \eta\left(2 - \eta_0L \right)D_f(x^t, x^*).	
	\end{align*}
	Now let us consider $\mu$-strongly convex $f$. We have
	\begin{align*}
		\|w^t - w^*\|^2
		&= \|x^t - x^*\|^2 - 2 \eta\<\nabla f(x^t) - \nabla f(x^*), x^t - x^*> + \eta^2\|\nabla f(x^t) - \nabla f(x^*)\|^2 \\
		&\overset{\eqref{eq:scal_prod_tight_str_cvx}}{\le} \left(1 - \frac{2 \eta \mu L}{L + \mu}\right)\|x - y\|^2 - \eta\left(\frac{2}{L + \mu} - \eta\right) \|\nabla f(x^t) - \nabla f(x^*)\|^2 \\
		&\overset{\eqref{eq:scal_prod_str_cvx}}{\le} \left(1 - \frac{2 \eta \mu L}{L + \mu}\right)\|x - y\|^2 - \eta\left(\frac{2}{L + \mu} - \eta\right) \mu^2\|x^t - x^*\|^2  \\
		&= (1 - \eta\mu)^2\|x^t - x^*\|^2 \\
		&\le (1 - \eta\mu)\|x^t - x^*\|^2.
	\end{align*}
	The last step simply uses $1 - \eta\mu\le 1$, which, of course, makes our guarantees slightly weaker, but, on the other hand, puts Gradient Descent under the umbrella of Assumption~\ref{as:method}.
\end{proof}

\subsection{Key lemma}\label{ap:key_lemma}
The  result below is  the most important lemma of our paper as it lies at the core of our analysis. It provides a very generic statement about the step with stochastic proximal operators.  At the same time, it is a mere corollary of firm non-expansiveness of the proximal operator.

\begin{lemma} \label{lem:key}
    Let $z^t = \proxR(w^t - \eta y^t)$ and $x^{t+1} = \proxj(z^{t} + \eta_j  y_j^t)$, where $j$ is sampled from $\{1,\dotsc, m\}$ with probabilities $\{p_1,\dotsc, p_m\}$, $\eta_j = \frac{\eta}{m p_j}$ and $\eta$ is a positive number. If $y_j^{t+1} = y_j^t + \frac{1}{\eta_j}(z^t - x^{t+1})$ and $y_k^{t+1} = y_k^t$ for all $k\neq j$, it holds
    \begin{align*}
        \EE \left[\|x^{t+1} - x^*\|^2 + \cY^{t+1} \right]
        \le \EE\left[\|w^t - w^*\|^2 + \left(1 - \frac{\gamma}{m(1+\gamma)}\right)\cY^t - \|z^t - w^t - (x^* - w^*)\|^2\right],
    \end{align*}
    where $\gamma \eqdef \min_{j=1,\dotsc,m} \frac{1}{\eta_j L_j}$ and $L_j\in\RR\cup \{+\infty\}$ is the smoothness constant of $g_j$.
\end{lemma}
\begin{proof}
    Mention that $x^* = \proxj(x^* + \eta_j  y_j^*)$ by optimality condition. In addition, it holds by definition $y_j^{t+1} = \frac{1}{\eta_j}(z^t + \eta_j y_j^t - x^{t+1})$, so property~\eqref{eq:nonexp} yields
    \begin{align*}
        \|x^{t+1} - x^*\|^2 + \left(1 + \frac{1}{\eta_j L_j}\right)\eta_j^2 \|y_j^{t+1} - y_j^*\|^2
        &\le \|z^t + \eta_j y_j^t - (x^* + \eta_j y_j^*)\|^2
    \end{align*}
    and we can replace $1+\frac{1}{\eta_j L_j}$ with $1+\gamma$ since $\gamma\le \frac{1}{\eta_j L_j}$.
    
    Let $\EE_j$ be the expectation with respect to sampling of $j$. Then, we observe
    \begin{align}
        &\EE_j \|z^t + \eta_j y_j^t - (x^* + \eta_j y_j^*)\|^2\notag \\
        &= \|z^t - x^*\|^2 + \EE_j\left[\frac{\eta^2}{m^2p_j^2}\|y_j^t - y_j^*\|^2\right] + 2\<z^t - x^*, \eta\EE_j \left[\frac{1}{mp_j}(y_j^t - y_j^*) \right]> \nonumber\\
        &= \|z^t - x^*\|^2 + \frac{\eta^2}{m^2}\sumkm \frac{1}{p_k}\|y_k^t - y_k^*\|^2 + 2\eta\<z^t - x^*, y^t - y^*>. \label{eq:firm_nonexp_for_stoch_prox}
    \end{align}
    Denote $w^*\eqdef x^* - \eta \nabla f(x^*)$. Another optimality condition from Lemma~\ref{lem:optimality} is $x^*=\proxR(w^* - \eta y^*)$, so let us use~\eqref{eq:nonexp} one more time to obtain
    \begin{align*}
        \|z^t - x^*\|^2
        &\le \|w^t - \eta y^t - (w^* - \eta y^*)\|^2 - \|w^t - \eta y^t - z^t - (w^* - \eta y^* - x^*)\|^2 \\
        &= \eta^2\|y^t - y^*\|^2 - 2\eta\<w^t - w^* , y^t - y^*> + \|w^t - w^*\|^2 - \|w^t - \eta y^t - z^t - (w^* - \eta y^* - x^*)\|^2.
    \end{align*}
    Furthermore,
    \begin{align*}
        \|w^t - \eta y^t - z^t - (w^* - \eta y^* - x^*)\|^2
        &= \|w^t - z^t - (w^* - x^*)\|^2\\
        &\quad - 2\eta\<w^t - z^t - (w^* - x^*), y^t - y^*> + \eta^2\|y^t - y^*\|^2,
    \end{align*}
    so
    \begin{align*}
        \|z^t - x^*\|^2 
        &\le - 2\eta\<w^t - w^* , y^t - y^*> + \|w^t - w^*\|^2 + 2\eta\<w^t - z^t - (w^* - x^*), y^t - y^*> \\
        &\quad - \|w^t - z^t - (w^* - x^*)\|^2 \\
        &=  \|w^t - w^*\|^2 - 2\eta\<z^t - x^*, y^t - y^*> - \|w^t - z^t - (w^* - x^*)\|^2.
    \end{align*}
    Together with the previously obtained bounds it adds up to
    \begin{align*}
        \EE_j \|z^t + \eta_j y_j^t - (x^* + \eta_j y_j^*)\|^2
        &\le \|w^t - w^*\|^2 + \frac{\eta^2}{m^2}\sumkm \frac{1}{p_k}\|y_k^t - y_k^*\|^2 - \|z^t - w^t - (x^* - w^*)\|^2.
    \end{align*}
    To get the expression in the left-hand side of this lemma's statement, let us add the missing sum and evaluate its expectation:
    \begin{align*}
        \EE \left[\sumkm \eta_k^2\|y_k^{t+1} - y_k^*\|^2 \right]
        = \EE \|y_j^{t+1} - y_j^*\|^2 + \EE \left[\sum_{k\neq j} \eta_k^2\|y_k^{t+1} - y_k^*\|^2 \right].
    \end{align*}
    Clearly, all summands in the last sum were not changed at iteration $t$, so
    \begin{align*}
        \EE_j \left[\sum_{k\neq j} \eta_k^2\|y_k^{t+1} - y_k^*\|^2 \right]
        &= \EE_j \left[\sum_{k\neq j} \eta_k^2\|y_k^{t} - y_k^*\|^2 \right]\\
        &= \sumkm (1 - p_k)\eta_k^2\|y_k^{t} - y_k^*\|^2 \\ 
        &= \sumkm \eta_k^2\|y_k^{t} - y_k^*\|^2 - \frac{\eta^2}{m^2}\sumkm \frac{1}{p_k}\|y_k^{t} - y_k^*\|^2.
    \end{align*}
    The negative sum will cancel out with the same in equation~\eqref{eq:firm_nonexp_for_stoch_prox} and we conclude the proof.
\end{proof}

\subsection{Convergence of Bregman divergence to 0 almost surely}\label{ap:almost_sure}
Here we formulate a result that we only briefly mentioned in the main text. It states that for convex problems, Bregman divergence $D_f(x^t, x^*)$ almost surely converges to 0. To show it, let us borrow the classical result on supermartingale  convergence.
\begin{proposition}[\cite{bertsekas2015convex}, Proposition~A.4.5]\label{pr:supmart}
	Let $\{X^t\}_t$, $\{Y^t\}_t$, $\{Z^t\}_t$ be three sequences of non-negative random variables and let $\{\cF^t\}_t$ be a sequence of $\sigma$-algebras such that $\cF^t\subset \cF^{t+1}$ for all $t.$ Assume that:
		\begin{itemize}
			\item The random variables $X^t, Y^t, Z^t$ are non-negative and $\cF^t$-measurable.
			\item For each $t$, we have $\EE[X^{t+1}\mid \cF^t] \le X^t - Y^t + Z^t$.
			\item There holds, with probability 1,
			\begin{align*}
				\sum_{t=0}^{\infty} Z^t < \infty.
			\end{align*}
		\end{itemize}
	Then $X^t$ converges to a non-negative random variable $X$ and we have $\sum_{t=0}^{\infty} Y^t < \infty$ with probability 1.
\end{proposition}

\begin{theorem}
	Take a method that satisfies Assumption~\ref{as:method}(a), a stepsize $\eta\le \eta_0$ and an optimum $x^*$ satisfying Assumption~\ref{as:optimality}. Then, with probability 1 it holds $D_f(x^t, x^*)\to 0$.
\end{theorem}
\begin{proof}
	Fix any solution $x^*$, $y_1^*, \dotsc, y_m^*$. Let $\cF^t=\sigma(x^0, y_1^0, \dotsc, y_m^0, \dotsc, x^t, y_1^t, \dotsc, y_m^t)$ be the $\sigma$-algebra generated by all random variables prior to moment $t$, and let $\overline \cM^t$ be $\cM^t$ conditioned on $\cF^t$, i.e., $\overline \cM^t \eqdef \cM^t | \cF^t$, from which it follows $\cM^t = \EE \overline \cM^t$. Then, the assumptions of Proposition~\ref{pr:supmart} are satisfied for sequences 
	\begin{align*}
		X^t 
		&= \|x^t - x^*\|^2 + \overline \cM^t + (1 + \gamma)\sumkm \eta_k^2\|y_k^t - y_k^*\|^2, \\
		Y^t
		&= \omega \eta D_f(x^t, x^*), \\
		Z^t
		&=0.
	\end{align*}
	Therefore, we have that $\sum_{t=0}^{\infty} Y^t < \infty$ and $Y^t \to 0$ almost surely, from which it follows $D_f(x^t, x^*)\to 0$.
\end{proof}
The almost sure guarantee is not applicable to SGD which has $\cM^0$ proportional to the number of iterations. We leave its analysis as well as analysis of convergence of $x^t$ to an optimum for future work.
\subsection{Proof of Theorem~\ref{th:1_over_t_rate} ($\cO\left(\nicefrac{1}{t}\right)$ rate)}\label{ap:1_t_rate}
Below we provide the proof of $\cO\left(\nicefrac{1}{t}\right)$ rate for general convex functions.
\begin{customthm}{\ref{th:1_over_t_rate}}
    Assume $f$ is $L$-smooth and $\mu$-strongly convex, $g_1, \dotsc, g_m, R$ are convex, closed and lower semi-continuous. Take a method satisfying Assumption~\ref{as:method} and $\eta\le \eta_0$, then
    \begin{align*}
        \EE D_f(\overline x^t, x^*)
        \le \frac{1}{\omega \eta t}\cL^0,
    \end{align*}
    where $\cL^0\eqdef \|x^0 - x^*\|^2 + \cM^0 + \sumkm  \eta_k^2 \|y_k^0 - y_k^*\|^2$ and $\overline x^t \eqdef \frac{1}{t}\sum_{k=0}^{t-1} x^k$.
\end{customthm}
\begin{proof}
    Recall that
    \begin{align*}
        \cL^t
        \eqdef  \EE\|x^t - x^*\|^2 + \cM^t + \cY^t,
    \end{align*}
    and by Assumption~\ref{as:method} combined with Lemma~\ref{lem:key}
    \begin{align*}
    		\cL^{t+1}
    		\le \cL^t - \omega \eta \EE D_f(x^t, x^*).
    \end{align*}
	Telescoping this inequality from $0$ to $t-1$, we obtain
    \begin{align*}
        \EE \left[\sum_{k=0}^{t-1} D_f(x^t, x^*)\right]
        \le \frac{1}{\omega\eta}(\cL^0 - \cL^{t})
        \le \frac{1}{\omega\eta}\cL^0.
    \end{align*}    
    By convexity of $f$, the left-hand side is lower bounded by $t\EE D_f(\overline x^t, x^*)$, so dividing both sides by $t$ finishes the proof.
\end{proof}

\subsection{Proof of Theorem~\ref{th:1_t2_rate} ($\cO(\nicefrac{1}{t^2})$ rate)}\label{ap:1_t2_rate}

In this subsection, we show the $\cO\left(\nicefrac{1}{t^2}\right)$ rate.
\begin{customthm}{\ref{th:1_t2_rate}}
	Consider updates with time-varying stepsizes, $\eta^t = \frac{2}{\omega\mu(a + t)}$ and $\eta_j^t = \frac{\eta^t}{m p_j}$ for $j=1,\dotsc, m$, where $a\ge 2\max\left\{\frac{1}{\omega\mu\eta_0}, \frac{1}{\rho} \right\}$. Then, it holds
	\begin{align*}
		\EE \|x^t - x^*\|^2
		\le \frac{a^2}{(t+a-1)^2}\cL^0,
	\end{align*}
	where $\cL^0 = \|x^0 - x^*\|^2 + \cM^0 + \sumkm (\eta_k^0)^2 \|y_k^0 - y_k^*\|^2$.
\end{customthm}
\begin{proof}
	For this proof, we redefine the sequence $\cY^t$ to have time-varying stepsizes:
	\begin{align*}
		\cY^t
		\eqdef \sumkm (\eta_k^{t})^2\EE\|y_k^{t} - y_k^*\|^2.
	\end{align*}
	Before writing a new recurrence, let us note that 
	\begin{align*}
		(1 - \omega\eta^t\mu)\left(\frac{\eta^{t-1}}{\eta^{t}}\right)^2
		= \frac{\left(1 - \frac{2}{a + t}\right)(a+t)^2}{(a+t-1)^2}
		= \frac{(a+t-2)(a+t)}{(a+t-1)^2}
		< 1,
	\end{align*}
	so $1 - \omega\eta^t\mu \le \left(\frac{\eta^t}{\eta^{t-1}}\right)^2$. Then, Lemma~\ref{lem:key} gives a similar recurrence to what we have seen in other proofs, but the stepsizes in the right-hand side are now time-dependent:
	\begin{align*}
		\cL^{t+1}
		&=\EE \|x^{t+1} - x^*\|^2 + \cM^{t+1} + \cY^{t+1}\\
		&\le (1 - \omega\eta^t \mu) \EE\|x^t - x^*\|^2 + (1 - \rho)\cM^t + \sumkm (\eta_k^{t})^2\EE\|y_k^{t} - y_k^*\|^2 \\
		&\le (1 - \omega\eta^t\mu)\EE\left[\|x^t - x^*\|^2 + \cM^t\right] + \left(\frac{\eta^t}{\eta^{t-1}}\right)\sumkm(\eta_k^{t-1})^2\EE\|y_k^t - y_k^*\|^2 \\
		&\le \left(\frac{\eta^t}{\eta^{t-1}}\right)^2\EE\left[\|x^t - x^*\|^2 + \cM^t\right] + \left(\frac{\eta^t}{\eta^{t-1}}\right)^2\sumkm(\eta_k^{t-1})^2\EE\|y_k^t - y_k^*\|^2  \\
		&= \left(\frac{\eta^t}{\eta^{t-1}}\right)^2\cL^t.
	\end{align*}
	Recursing this inequality yields
	\begin{align*}
		\cL^{t+1} 
		\le \cL^0\prod_{k=1}^t\left(\frac{\eta^k}{\eta^{k-1}}\right)^2
		= \left(\frac{\eta^t}{\eta^{0}}\right)^2\cL^0
		= \left(\frac{a}{a+t}\right)^2\cL^0.
	\end{align*}
\end{proof}
\subsection{Proof of Theorem~\ref{th:sgd_str_cvx} ($\cO(\nicefrac{1}{t})$ rate of SGD)}\label{ap:sgd_str_cvx}
Here we consider the case where $f(x)$ is given as expectation parameterized by a random variable~$\xi$,
\begin{align*}
	f(x)
	= \EE_\xi f_\xi(x).
\end{align*}
While it is often assumed in the literature that $\EE \|\nabla f_\xi(x) - \nabla f(x)\|^2\le \sigma^2$ uniformly over $x$, we do not need this assumption and bound the variance using the following lemma.
\begin{lemma}\label{lem:sgd_variance}
	Let $w^t = x^t - \eta \nabla f_{\xi^t}(x^t)$, where random function $f_\xi(x)$ is almost surely convex and $L$-smooth. Then,
	\begin{align}
		\EE \|\nabla f_{\xi^t}(x^t) - \nabla f(x^*)\|^2
		\le 4L \EE D_f(x^t, x^*) + 2\sigma_*^2,\label{eq:sgd_variance}
	\end{align}
	where $\sigma_*^2\eqdef \EE \|\nabla f_{\xi}(x^*) - \nabla f(x^*)\|^2$, i.e., $\sigma_*^2$ is the variance at an optimum. If more than  one $x^*$ exists, take the one that minimizes $\sigma_*^2$.
\end{lemma}
\begin{proof}
	This proof is based on existing results for SGD and goes in a very standard way. By Young's inequality
	\begin{align*}
		\EE \|\nabla f_{\xi^t}(x^t) - \nabla f(x^*)\|^2
		&\le 2 \EE \|\nabla f_{\xi^t}(x^t) - \nabla f(x^*; \xi^t)\|^2 + 2 \EE \|\nabla f_{\xi^t}(x^*)  - \nabla f(x^*)\|^2 \\
		&\overset{\eqref{eq:grad_dif_bregman}}{\le} 4L \EE D_{f_{\xi^t} }(x^t, x^*) + 2\sigma_*^2 \\
		&= 4L \EE D_{f}(x^t, x^*) + 2\sigma_*^2.
	\end{align*}
\end{proof}
In the proof of Theorem~\ref{th:sgd_str_cvx} we will again need time-varying stepsize and $\cY^t$ should be defined as
\begin{align*}
		\cY^t
		\eqdef \sumkm (\eta_k^{t})^2\EE\|y_k^{t} - y_k^*\|^2.
	\end{align*}
But before let us prove a simple statement about sequences with contraction and additive error.
\begin{lemma}
\label{lem:recurrence_with_sigma}
	Assume that sequence $\{\cL^t\}_t$ satisfies inequality $\cL^{t+1}\le \left(\frac{\eta^t}{\eta^{t-1}}\right)^2 \cL^t + 2(\eta^t)^2 \sigma_*^2$ with some constant $\sigma_*\ge 0$. Then, it holds
	\begin{align*}
		\cL^{t} 
		\le \left(\frac{\eta^{t-1}}{\eta^0}\right)^2 \cL^0 + 2t (\eta^{t-1})^2 \sigma_*^2.
	\end{align*}
\end{lemma}
\begin{proof}
	We will prove the inequality by induction. For $t=0$ it is straightforward. The induction step follows from
	\begin{align*}
		\cL^{t+1}
		&\le \left(\frac{\eta^t}{\eta^{t-1}}\right)^2 \cL^t + 2(\eta^t)^2 \sigma_*^2 \\
		&\le \left(\frac{\eta^t}{\eta^{t-1}}\right)^2\left(\frac{\eta^{t-1}}{\eta^0}\right)^2 \cL^0 + 2\left(\frac{\eta^t}{\eta^{t-1}}\right)^2(\eta^{t-1})^2t\sigma_*^2+ 2t (\eta^{t-1})^2 \sigma_*^2 \\
		&= \left(\frac{\eta^{t}}{\eta^0}\right)^2 \cL^0 + 2(t+1) (\eta^{t-1})^2 \sigma_*^2.
	\end{align*}
\end{proof}
Now we are ready to prove the theorem.
\begin{customthm}{\ref{th:sgd_str_cvx}}
	Assume $f$ is $\mu$-strongly convex, $f(\cdot; \xi)$ is almost surely convex and $L$-smooth. Let the update be produced by SGD, i.e., $v^t = \nabla f(x^t; \xi^t)$, and let us use time-varying stepsizes $\eta^{t-1} = \frac{2}{a + \mu t}$ with $a\ge 4L$. Then, it holds
	\begin{align*}
		\EE \|x^t - x^*\|^2
		\le \frac{8\sigma_*^2}{\mu(a + \mu t)} + \frac{a^2}{(a + \mu t)^2}\cL^0.
	\end{align*}
\end{customthm}
\begin{proof}
	It holds by Lemma~\ref{lem:sgd_variance}
	\begin{align*}
		\EE \|\nabla f_{\xi^t}(x^t)  - \nabla f(x^*)\|^2
		\le 4L \EE D_f(x^t, x^*) + 2\sigma_*^2.
	\end{align*}
	Therefore, for $w^t \eqdef x^t - \eta^t v^t = x^t - \eta^t \nabla f_{\xi^t}(x^t) $ and $w^*\eqdef x^* - \eta^t \nabla f(x^*)$ we have
	\begin{align*}
		\EE \|w^t - w^*\|^2
		&= \EE \left[\|x^t - x^*\|^2 - 2\eta^t\<\nabla f(x^t) - \nabla f(x^*), x^t - x^*> + (\eta^t)^2\|\nabla f_{\xi^t}(x^t) - \nabla f(x^*)\|^2 \right] \\
		&\overset{\eqref{eq:sgd_variance}}{\le} \EE \left[\|x^t - x^*\|^2 - 2\eta^t\<\nabla f(x^t) - \nabla f(x^*), x^t - x^*> + 4(\eta^t)^2LD_f(x^t, x^*) + 2(\eta^t)^2\sigma_*^2 \right]\\
		&\overset{\eqref{eq:scal_prod_cvx}}{\le} \EE\Bigl[(1 - \eta^t\mu)\|x^t - x^*\|^2 - 2\eta^t(\underbrace{1 - 2\eta^t L}_{\ge 0})D_f(x^t, x^*) + 2(\eta^t)^2\sigma_*^2 \Bigr] \\
		&\le (1 - \eta^t\mu)\EE\|x^t - x^*\|^2 + 2(\eta^t)^2\sigma_*^2.
	\end{align*}
	Using the same argument as in the proof of Theorem~\ref{th:1_t2_rate}, we can show that $1 - \eta^t\mu\le \left(\frac{\eta^t}{\eta^{t-1}}\right)^2$. Combining these results with Lemma~\ref{lem:key}, we obtain for $\cL^{t+1}\eqdef \EE\|x^{t+1} - x^*\|^2 + \cY^{t+1}$
	\begin{align*}
		\cL^{t+1}
		&\le (1 - \eta^t\mu)\EE\|x^t - x^*\|^2 + \sumkm (\eta_k^{t})^2\EE\|y_k^{t} - y_k^*\|^2 + 2(\eta^t)^2\sigma_*^2 \\
		&\le \left(\frac{\eta^t}{\eta^{t-1}}\right)^2\EE\|x^t - x^*\|^2 + \left(\frac{\eta^t}{\eta^{t-1}}\right)^2\cY^t + 2(\eta^t)^2\sigma_*^2 \\
		&= \left(\frac{\eta^t}{\eta^{t-1}}\right)^2\cL^t + 2(\eta^t)^2\sigma_*^2.
	\end{align*}
	By Lemma~\ref{lem:recurrence_with_sigma}
	\begin{align*}
		\EE\|x^t - x^*\|^2
		\le \cL^t
		\le \left(\frac{\eta^{t-1}}{\eta^0}\right)^2 \cL^0 + 2t (\eta^{t-1})^2 \sigma_*^2
		\le \frac{a^2}{(a + \mu t)^2}\cL^0 + \frac{8t}{(a + \mu t)\mu t} \sigma_*^2.
	\end{align*}
\end{proof}

\subsection{Proof of Theorem~\ref{th:lin_conv_lin_model} (linear rate for $g_j=\phi_j(\mA_j^\top x)$)}\label{ap:lin_conv_lin_model}

Let us now show linear convergence of our method when the consider problem has linear structure, i.e., $g_j(x) = \phi_j(\mA_j^\top x)$.

We first need a lemma on the nature of $y_1^t,\dotsc, y_m^t$ in the considered case.

\begin{lemma}\label{lem:linear}
    Let the proximal sum be of the form $\frac{1}{m}\sumjm \phi_j(\mA_j^\top x)$ with some matrices $\mA_j\in \RR^{d\times d_j}$, and $y_j^0 = \mA_j\beta_j^0$ for $j=1,\dotsc, m$. Then, if Assumption~\ref{as:optimality2} is satisfied, for any $t$ and $j$ we have
    \begin{align*}
        y_j^t = \mA_j \beta_j^t, \quad  y^t = \frac{1}{m} \sumjm y_j^t= \frac{1}{m}\mA \beta^t, \quad    y_j^* =  \mA_j\beta_j^*, \quad  y^* = \frac{1}{m}\sumjm y_j^*= \frac{1}{m}\mA \beta^*
    \end{align*}
    with some vectors $\beta_i^t, \beta_i^*\in\RR^{d_i}$ with $i=1,\dotsc, m$, $\beta^t\eqdef ((\beta_1^t)^\top, \dotsc, (\beta_m^t)^\top)^\top$, $\beta^* \eqdef ((\beta_1^*)^\top, \dotsc, (\beta_m^*)^\top)^\top$ and $\mA\eqdef [\mA_1, \dotsc, \mA_m]$.
\end{lemma}
\begin{proof}
    By definition $y_j^{t+1} = y_j^t + \frac{1}{\eta_j}(z^t - x^{t+1}) = \frac{1}{\eta_j}(z^t + \eta_j y_j^t - x^{t+1})$. In addition, by Lemma~\ref{lem:prox_of_composition} there exists $\beta_j^{t+1} \in \partial \phi_j (\mA_j^\top x^{t+1})$ such that $x^{t+1} = \proxj(z^t + \eta_j y_j^t) \in z^t + \eta_j y_j^t - \eta_j\mA_j \beta_j^{t+1}$
     and, thus, $y_j^{t+1} = \mA_j\beta_j^{t+1} $. Therefore, we also have $y^t = \frac{1}{m}\mA\beta^t$.
    
    The claims about $y_1^*,\dotsc, y_m^*$  and $y^*$ follow from Assumption~\ref{as:optimality2}.
\end{proof}

Now it is time to prove Theorem~\ref{th:lin_conv_lin_model}.

\begin{customthm}{\ref{th:lin_conv_lin_model}}
    Assume that $f$ is $\mu$-strongly convex, $R\equiv 0$, $g_j(x) = \phi_j(\mA_j^\top x)$ for $j=1,\dotsc, m$ and take a method satisfying Assumption~\ref{as:method} with $\rho>0$. Then, if $\eta\le \eta_0$,
    \begin{align*}
        \EE \|x^t - x^*\|^2
        \le \left(1 - \min\{\rho, \omega\eta\mu, \rho_{A}\} \right)^t \cL^0,
    \end{align*}
    where $\rho_{A} \eqdef \lambda_{\min}(\mA^\top\mA) \min_j \left(\frac{p_j}{\|\mA_j\|}\right)^2$, and $\cL^0\eqdef \|x^0 - x^*\|^2 + \cM^0 + \sumkm  \eta_k^2 \|y_k^0 - y_k^*\|^2$.
\end{customthm}

\begin{proof}
    Lemma~\ref{lem:key} with Assumption~\ref{as:method} yields
    \begin{align*}
        \cL^{t+1}
        &\le (1 - \min\{\rho, \omega\eta\mu\}) \left(\EE \|x^t - x^*\|^2 + \cM^t\right)  + \cY^t - \EE \|z^t - w^t - (x^* - w^*)\|^2.
    \end{align*}
    By Lemma~\ref{lem:linear}
    \begin{align*}
        \cY^t 
        = \sumkm  \eta_k^2 \EE\|y_k^t - y_k^*\|^2
        = \sumkm \eta_k^2 \EE\|\mA_k(\beta_k^t - \beta_k^*)\|^2.
    \end{align*}
    Since we assume $R\equiv \mathrm{const}$, we have $z^t - w^t = x^t - \eta v^t - \eta y^t - (x^t - \eta v^t) = -\eta y^t$ and
    \begin{align*}
        \|z^t - w^t - (x^* - w^*)\|^2
        &= \eta^2 \|y^t - y^*\|^2\\
        &= \frac{\eta^2}{m^2} \|\mA (\beta^t - \beta^*)\|^2\\
        &\ge \lambda_{\min} (\mA^\top\mA) \frac{\eta^2}{m^2}\|\beta^t - \beta^*\|^2\\
        &= \lambda_{\min} (\mA^\top\mA) \sumkm \frac{p_k^2}{\|\mA_k\|^2}\eta_k^2\|\mA_k\|^2\|\beta_k^t - \beta_k^*\|^2\\
        &\ge \lambda_{\min} (\mA^\top\mA) \min_k \frac{p_k^2}{\|\mA_k\|^2} \sumkm \eta_k^2\|\mA_k\|^2\|\beta_k^t - \beta_k^*\|^2\\
        &\ge \lambda_{\min} (\mA^\top\mA) \min_k \frac{p_k^2}{\|\mA_k\|^2} \sumkm \eta_k^2\|\mA_k(\beta_k^t - \beta_k^*)\|^2\\
        &=\rho_A \sumkm  \eta_k^2 \|y_k^t - y_k^*\|^2.
    \end{align*}
    Therefore,
    \begin{align*}
        \cY^{t} - \EE \|z^t - w^t - (x^* - w^*)\|^2
        &\le (1 - \rho_A)\cY^t.
    \end{align*}
    Putting the pieces together, we obtain
    \begin{align*}
        \cL^{t+1}
        \le (1 - \min\{\rho, \omega\eta\mu, \rho_A\})\cL^t,
    \end{align*}
    from which it follows that $\cL^t$ converges to 0 linearly. Finally, note that $\EE \|x^t - x^*\|^2\le \cL^t\le (1 - \min\{\rho, \omega\eta\mu, \rho_A\})^t\cL^0$.
\end{proof}

\subsection{Proof of Theorem~\ref{th:lin_constr} (linear constraints)}\label{ap:lin_constr}
Here we discuss the problem of linearly constrained minimization
\begin{align*}
	\min_x \{f(x): \mA^\top x = b \}.
\end{align*}
We split matrix $\mA =  [\mA_1, \dots, \mA_m]$ and vector $b=(b_1^\top, \dotsc, b_m^\top)^\top$ and define projection operator $\Pi_j(\cdot)\eqdef \Pi_{\{x:\mA_j^\top x= b_j\}}(\cdot)$ . Since $y_j^t\in \Range{\mA_j}$, it is orthogonal to the hyperplane $\{x: \mA_j^\top x = b_j\}$ for any $x$ it holds
\begin{align*}
	\Pi_j(x + y_j^t) 
	=\Pi_j(x).
\end{align*}
This allows us to write a memory-efficient version of Algorithm~\ref{alg:sdm} as given in Algorithm~\ref{alg:sdm_lin_system}. If only a subset of functions $g_1, \dotsc, g_m$ is linear equality constraints, then similarly the corresponding vectors $y_j^t$ are not needed in the update, although they are still useful for the analysis.
\begin{algorithm}[t]
   \caption{Stochastic Decoupling Method for linearly constrained problem.}
   \label{alg:sdm_lin_system}
\begin{algorithmic}[1]
   \Require Stepsize $\eta$, initial vectors $x^0, y^0\in \RR^d$, probabilities $p_1,\dotsc, p_m$, oracle that gives gradient estimates
   \For{$t=0,1,\dotsc$}
	   \State Produce an estimate $v^t$ of $\nabla f(x^t)$
	   \State $z^{t} = \proxR(x^t - \eta v^t - \eta y^t)$
	   \State Sample $j$ from $\{1,\dotsc, m\}$ with probabilities $\{p_1, \dotsc, p_m\}$
	   \State $x^{t+1} = \Pi_{j}(z^t)$
	   \State $y^{t+1} = y^t + \frac{p_j}{\eta}(z^t - x^{t+1})$
   \EndFor
\end{algorithmic}
\end{algorithm}

Here we show that if $f$ is strongly convex and the non-smooth part is constructed of linear constraints, then we can guarantee linear rate of convergence. Moreover, the rate will depend only on the smallest nonzero eigenvalue of $\mA^\top \mA$, implying that even if $\mA^\top \mA$ is degenerate, convergence will be linear.
\begin{customthm}{\ref{th:lin_constr}}
	Under the same assumptions as in Theorem~\ref{th:lin_conv_lin_model} and assuming, in addition, that $g_j(x) = \ind_{\{x\;:\;\mA_j^\top x = b_j\}}$ it holds $\EE \|x^t - x^*\|^2 \le (1 - \min\{\rho, \omega\eta\mu, \rho_A\})^t\cL^0$ with $\rho_A=\lambda_{\min}^+(\mA^\top\mA)\min_{j}\left(\frac{p_j}{\|\mA_j\|}\right)^2$, i.e., $\rho_A$ depends only on  the smallest \textbf{positive} eigenvalue of $\mA^\top\mA$.
\end{customthm}
\begin{proof}
	The main reason we get an improved guarantee for linear constraints is that one can write a closed form expression for the proximal operators:
	\begin{align*}
		\proxj(x)
		= x - \mA_j(\mA_j^\top \mA_j)^\dagger (\mA_j^\top x - b_j).
	\end{align*}
	Assume that $j$ was sampled at iteration $t$, then
	\begin{align*}
		y_j^{t+1}
		&= \frac{1}{\eta_j}\left(z^t + \eta_j y_j^t - \proxj(z^t + \eta_j y_j^t) \right)\\
		&= \mA_j(\mA_j^\top \mA_j)^\dagger (\mA_j^\top (z^t + \eta_i y_j^t) - b_j).
	\end{align*}
	Therefore, for any $j$ and $t$ there exists a vector $x_j^t\in\RR^d$ such that 
	\begin{align*}
		y_j^{t} 
		&= \mA_j(\mA_j^\top \mA_j)^\dagger (\mA_j^\top x_j^t - b_j) \\
		&= \mA_j(\mA_j^\top \mA_j)^\dagger \mA_j^\top (x_j^t - x^*),
	\end{align*} 
	where the second step is by the fact that $x^*$ is from the set $\{x:\mA_j^\top x = b_j\}$.
	One can show using SVD that $\Range{(\mA_j^\top \mA_j)^\dagger\mA_j^\top}=\Range{\mA_j^\top}$. Then, $y_j^t - y_j^* = \frac{1}{m}\mA_j(\beta_j^t - \beta_j^*)$ with $\beta_j^t - \beta_j^*\in \Range{\mA_j^\top}$. This, in  turn, implies $\beta^t - \beta^* \in\Range{\mA^\top}$, so
    \begin{align*}
        \|z^t - w^t - (x^* - w^*)\|^2
        &= \eta^2 \|y^t - y^*\|^2\\
        &= \frac{\eta^2}{m^2} \|\mA (\beta^t - \beta^*)\|^2\\
        &\ge \lambda_{\min}^+ (\mA^\top\mA) \frac{\eta^2}{m^2}\|\beta^t - \beta^*\|^2.
    \end{align*}
    The rest of the proof goes the same way as that of Theorem~\ref{th:lin_conv_lin_model} in Appendix~\ref{ap:lin_conv_lin_model}.
\end{proof}

\subsection{Proof of Theorem~\ref{th:lin_conv_smooth} (smooth $g_j$)}\label{ap:lin_conv_smooth}
This is the only proof where Lemma~\ref{lem:key} is used with finite smoothness constants, i.e., $\max_{j=1,\dotsc, m}L_j < +\infty$. On the other hand, we will not use the negative square term from Lemma~\ref{lem:key}, which is rather needed in the case $g_j(x) = \phi_j(\mA_j^\top x)$.
\begin{customthm}{\ref{th:lin_conv_smooth}}
    Assume that $f$ is $L$-smooth and $\mu$-strongly convex, $g_j$ is $L_j$-smooth for $j=1,\dotsc, m$ and Assumption~\ref{as:method}(b) is satisfied. Then, Algorithm~\ref{alg:sdm} converges as
	\begin{align*}
		\EE \|x^t - x^*\|^2
		\le \left(1 - \min\left\{\omega\eta\mu, \rho, \frac{\gamma}{m(1+\gamma)}\right\}\right)^t \cL^0,
	\end{align*}	    
    where $\gamma \eqdef \min_{j=1,\dotsc,m} \frac{1}{\eta_j L_j}$.
\end{customthm}
\begin{proof}
	Following the same lines as in the proof of Theorem~\ref{th:lin_conv_lin_model}, we get a contraction in $\cY^t$. Now we obtain it from the fact that functions $g_1,\dotsc, g_m$ are smooth, so the recursion is
	\begin{align*}
		\cL^{t+1}
		&\le (1 - \omega\eta\mu)\EE\|x^t-x^*\|^2 +  (1 - \rho)\cM^t + \left(1 - \frac{\gamma}{m(1+\gamma)}\right)\cY^t \\
		&\le \left(1 - \min\left\{\omega\eta\mu, \rho, \frac{\gamma}{m(1+\gamma)}\right\}\right)\cL^t.
	\end{align*}
	This is sufficient to show the claimed result.
\end{proof}

\subsection{Proof of Corollary~\ref{cor:acc_in_g} (optimal stepsize)}\label{ap:acc_in_g}
Corollary~\ref{cor:acc_in_g} is a statement about the optimal stepsizes for the case where $g_1, \dotsc, g_m$ are smooth functions. Its proof is a mere check that the choice of stepsizes gives the claimed complexity. 
\begin{customcor}{\ref{cor:acc_in_g}}
	Choose as solver for $f$ SVRG or SAGA without minibatching, which satisfy Assumption~\ref{as:method} with $\eta_0=\frac{1}{5L}$ and $\rho = \frac{1}{3n}$, and consider for simplicity situation where $L_1= \dotsb = L_m \eqdef L_g$ and $p_1=\dotsb=p_m$. Define $\eta_{best} \eqdef \frac{1}{\sqrt{\omega \mu m L_g}}$,
	and set the stepsize to be $\eta=\min\{\eta_0, \eta_{best}\}$. Then the complexity to get $\EE\|x^t - x^*\|^2\le \varepsilon$ is
$
		\cO\left(\left(n + m + \frac{L}{\mu} + \sqrt{m\frac{L_g}{\mu}}\right) \log\frac{1}{\varepsilon}\right).
$
\end{customcor}
\begin{proof}
	According to Theorem~\ref{th:lin_conv_smooth}, in general, for any $\eta\le \eta_0$ the complexity to get $\EE\|x^t - x^*\|^2\le \varepsilon$ is 
	\begin{align*}
		\cO\left(\left(\frac{1}{\rho} + m + \frac{1}{\omega\eta\mu} + \frac{1}{\gamma}m\right) \log\frac{1}{\varepsilon}\right),
	\end{align*}
	where $\frac{1}{\gamma}$ simplifies to $\eta L_g$ when $L_1=\dotsb=L_m=L_g$ and $p_1=\dotsb=p_m=\frac{1}{m}$. In addition, for SVRG and SAGA, $\omega$ is a constant close to 1, so we can ignore it. Since $m$ and $\frac{1}{\rho}=3n$ do not depend on $\eta$, we only need to simplify the other two terms. One of them decreases with $\eta$ and the other increases, so the best complexity is achieved when the two quantities are equal to each other. The corresponding equation is
		$\omega \eta^2 \mu m L_g = 1$,
	whose solution is
	\begin{align*}
		\eta
		=\eta_{best}
		= \frac{1}{\sqrt{\omega \mu m L_g}}.
	\end{align*}
	Thus, we see that $\eta_{best}$ is optimal. Moreover, if $\eta_{best}\le \eta_0$ and $\eta=\eta_{best}$, the two terms in the complexity both become equal to
	\begin{align*}
		\frac{1}{\omega \eta_{best}\mu}
		= m \eta_{best} L_g
		= \sqrt{\frac{m L_g}{\omega \mu}}.
	\end{align*}
	However, if $\eta_{best} > \eta_0$, then $\eta=\min\{\eta_0, \eta_{best}\}=\eta_0$ is relatively small and the dominating term in the complexity is $\frac{1}{\omega\eta\mu}$ rather than $\eta L_g m$. Therefore, the complexity is
	\begin{align*}
		\cO\left(n + m + \frac{1}{\eta_0\mu} \right)
		= \cO\left(n + m + \frac{L}{\mu} \right).
	\end{align*}
	Combining the two complexities into one, we get the result.
\end{proof}
\subsection{Proof of Lemma~\ref{lem:svrg_saga} (SVRG and SAGA)}\label{ap:svrg_saga}
\begin{algorithm}[t]
   \caption{Stochastic Decoupling Method with SVRG.}
   \label{alg:svrg}
\begin{algorithmic}[1]
   \Require Stepsize $\eta$, initial vectors $x^0$, $u^0$, $\nabla f(u^0)$, $y_1^0, \dotsc, y_m^0$, $y^0=\avejm y_j^0$, minibatch size $\tau$
   \For{$t=0,1,\dotsc$}
	   \State Sample subset $S$ from $\{1,\dotsc, n\}$ of size $\tau$
	   \State $v^t= \frac{1}{\tau}\sum_{i\in S}\left(\nabla f_i(x^t) - \nabla f_i(u^t) + \nabla f(u^t)\right)$
	   \State $z^{t} = \proxR(x^t - \eta v^t - \eta y^t)$
	   \State $u^{t+1}= \begin{cases} x^t, &\text{with probability } \frac{\tau}{n},\\ u^t, & \text{with probability } 1 - \frac{\tau}{n} \end{cases}$
	   \State Sample $j$ from $\{1,\dotsc, m\}$ with probabilities $\{p_1, \dotsc, p_m\}$ and set $\eta_j = \frac{\eta}{mp_j}$
	   \State $x^{t+1} = \proxj\left(z^t + \eta_j y_j^t\right)$
	   \State $y_j^{t+1} = y_j^t + \frac{1}{\eta_j}(z^{t} - x^{t+1})$
	   \State $ y^{t+1} =  y^t + \frac{1}{m}(y_j^{t+1} - y_j^t)$
   \EndFor
\end{algorithmic}
\end{algorithm}
Here we consider the update rule of SVRG and SAGA with minibatch of size $\tau$. Following~\cite{hofmann2015variance}, we analyze SVRG and SAGA together by treating them both as memorization methods. More precisely, SAGA stores each gradient estimate, $\nabla f_i(u_i^t)$ individually, and SVRG stores only the reference point, $u^t$, itself and every iteration reevaluates $\nabla f_i(u^t)$ for all sampled $i$ to compute $v^t$. To avoid any confusion, we provide the explicit formulation of our method with the SVRG solver in Algorithm~\ref{alg:svrg}.

First of all, let us show that the estimate that we use, $v^t$, is unbiased.
\begin{lemma}
	Let us sample a set of indices $S$ of size $\tau$ from $\{1, \dotsc, n\}$. Then, it holds for
    \begin{align*}
	    v^t
	    \eqdef \frac{1}{\tau}\sum_{i\in S} \left(\nabla f_i(x^t) - \nabla f_i(u_i^t) + \alpha^t\right)
    \end{align*} 
    that it is unbiased
    \begin{align}
	    \EE v^t
	    = \EE\nabla f(x^t). \label{eq:svrg_saga_unbiased}
    \end{align} 
\end{lemma}
\begin{proof}
	Clearly, since $i$ is sampled with probability $\frac{\tau}{n}$, it holds
	\begin{align*}
		\frac{1}{\tau}\EE \left[\nabla f_i(x^t) - \nabla f_i(u_i^t)\right]
		&= \EE\left[\frac{1}{n}\sumkn (\nabla f_k(x^t) - \nabla f_k(u_k^t)) \right]\\
		&= \EE\left[\nabla f(x^t) - \avekn \nabla f_k(u_k^t) \right].
	\end{align*}
	Therefore, $\EE v^t  = \EE \nabla f(x^t)$. 
\end{proof}

We continue our analysis with the following lemma.
\begin{lemma}
	Consider SVRG and SAGA solver for $f$. Assume that every $f_i$ is convex and $L$-smooth and define
	\begin{align}
		\cM^t
		\eqdef \frac{3\eta^2}{\tau}\sumin \EE\|\nabla f_i(u_i^t) - \nabla f_i(x^*)\|^2,\label{eq:mt_svrg_saga}
	\end{align}
	where for SVRG $u_1^t=\dotsb=u_n^t$ is the reference point at moment $t$ and for SAGA it is the point whose gradient is saved in memory for function $f_i$. Then,
	\begin{align*}
		\cM^{t+1}
		\le \left(1 - \frac{\tau }{n}\right) \cM^t + 6\eta^2L \EE D_f(x^t, x^*).
	\end{align*}
\end{lemma}
\begin{proof}
	We have for SVRG that $\cM^{t+1}$ changes with probability $\frac{\tau}{n}$ and with probability $1 - \frac{\tau}{n}$ it remains the same. Then,
	\begin{align*}
		\EE \sumkn \|\nabla f_k(u_k^{t+1}) - \nabla f_k(x^*)\|^2
		&= \frac{\tau}{n} \sumkn \EE\|\nabla f_k(x^{t}) - \nabla f_k(x^*)\|^2 + \left(1  - \frac{\tau}{n}\right)\cM^t.
	\end{align*}
	Similarly, for SAGA we update exactly $\tau$ out of $n$ gradient in the memory, which leads to the following identity:
    \begin{align*}
        &\EE \sumkn \|\nabla f_k(u_k^{t+1}) - \nabla f_k(x^*)\|^2 \\
        &\qquad = \EE \sum_{i\in S}\|\nabla f_i(u_i^{t+1}) - \nabla f_i(x^*)\|^2 + \EE \sum_{i\not \in S} \|\nabla f_i(u_i^{t+1}) - \nabla f_i(x^*)\|^2\\
        &\qquad = \frac{\tau}{n}\sumkn\EE\|\nabla f_k(x^t) - \nabla f_k(x^*)\|^2 + \left(1  - \frac{\tau}{n}\right)\cM^t.
    \end{align*}
    In both cases, we obtained the same recursion. Now let us bound the gradient difference in the identity above:
    \begin{align*}
    		\frac{1}{n}\sumkn\|\nabla f_k(x^t) - \nabla f_k(x^*)\|^2
    		\overset{\eqref{eq:grad_dif_bregman}}{\le} \frac{1}{n}\sumkn 2L D_{f_k}(x^t, x^*) 
    		= 2L D_f(x^t, x^*).
    \end{align*}
    This gives us the claimed inequality.
\end{proof}
Now let us show how the recursion looks like when $\cM^{t+1}$ is combined with $\|w^{t} - w^*\|^2$.
\begin{lemma}\label{lem:saga}
    Consider the iterates of Algorithm~\ref{alg:sdm} with SVRG or SAGA estimate $v^t$. Let $f_1, \dotsc, f_n$ be convex and $L$-smooth, $f$ be $\mu$-strongly convex, $S$ be a subset of $\{1,\dotsc, n\}$ of size $\tau$ sampled with equal probabilities, $\alpha^t = \avein \nabla f_i(u_i^t)$ and $w^t = x^t - \eta v^t$ with 
    \begin{align*}
	    v^t
	    =\frac{1}{\tau} \sum_{i\in S} \left(\nabla f_i(x^t) - \nabla f_i(u_i^t) + \alpha^t\right)
    \end{align*} 
    then we have
    \begin{align*}
        \EE \left[\|w^{t} - w^*\|^2 + \cM^{t+1} \right]
        \le (1 - \rho)\EE\left[\|x^t - x^*\|^2 + \cM^t\right],
    \end{align*}
    where $w^* \eqdef x^* - \eta \nabla f(x^*)$ and  $\rho\eqdef \min \left\{\eta \mu, \frac{\tau}{3n}\right\}$.
\end{lemma}
\begin{proof}
    It holds
    \begin{align*}
        \EE \|w^{t} - w^*\|^2
        &= \EE \|x^t - x^* - \eta(v^t - \nabla f(x^*))\|^2 \\
        &= \EE\left[\|x^t - x^*\|^2 - 2\eta\<x^t - x^*, \EE [v^t\mid x^t] - \nabla f(x^*)> + \eta^2 \|v^t - \nabla f(x^*)\|^2 \right]\\
        &\overset{\eqref{eq:svrg_saga_unbiased}}{=} \EE\left[\|x^t - x^*\|^2 - 2\eta\<x^t - x^*, \nabla f(x^t) - \nabla f(x^*)> + \eta^2 \|v^t - \nabla f(x^*)\|^2 \right].\\
        &\overset{\eqref{eq:scal_prod_cvx}}{\le}
         \EE\left[(1 - \eta \mu)\|x^t - x^*\|^2 - 2\eta D_f(x^t, x^*) + \eta^2 \|v^t - \nabla f(x^*)\|^2 \right].
    \end{align*}
    On the other hand, by Jensen's and Young's inequalities
    \begin{align*}
        \EE \|v^t - \nabla f(x^*)\|^2 
        &= \EE\left\|\frac{1}{\tau}\sum_{i\in S}(\nabla f_i(x^t) - \nabla f_i(u_i^t)) + \alpha^t - \nabla f(x^*)) \right\|^2\\
        &\le \frac{1}{\tau}\EE\sum_{i\in S}\left\|\nabla f_i(x^t) - \nabla f_i(x^*) + \nabla f_i(x^*)- \nabla f_i(u_i^t) + \alpha^t - \nabla f(x^*) \right\|^2\\
        &\le \frac{2}{\tau}\EE\sum_{i\in S}\left\|\nabla f_i(x^t) - \nabla f_i(x^*)\right\|^2 + \frac{2}{\tau}\EE\sum_{i\in S}\|\nabla f_i(x^*) - \alpha_i^t + \alpha^t - \nabla f(x^*)\|^2 \\
        &\overset{\eqref{eq:grad_dif_bregman}}{\le} 4L\EE D_f(x^t, x^*) + \frac{2}{n}\sumkn\EE\left\|\nabla f_k(x^*) - \alpha_k^t + \alpha^t - \nabla f(x^*) \right\|^2.
    \end{align*}
    Using inequality $\EE\|X - \EE X\|^2 \le \EE \|X\|^2$ that holds for any random variable $X$, the second term can be simplified to
    \begin{align*}
        &\frac{2}{n}\sumkn\EE\left\|\nabla f_k(x^*) - \alpha_k^t + \alpha^t - \nabla f(x^*) \right\|^2 \\
        &\le \frac{2}{n}\sumkn\EE\|\nabla f_k(u_k^t) - \nabla f_k(x^*)\|^2 \\
        &= \frac{2\tau }{3n}\cM^t.
    \end{align*}
    Thus,
    \begin{align}
        \EE \left[\|w^{t} - w^*\|^2 + \cM^{t+1} \right]
        &\le (1 - \eta \mu)\EE\|x^t - x^*\|^2 - 2\eta\left(1 - 2\eta L - 3\eta  L\right)\EE D_f(x^t, x^*)\notag \\
        &\quad + \left(\left(1 - \frac{\tau}{n}\right) + \frac{2\tau }{3n}\right) \cM^t. \label{eq:saga_precise}
    \end{align}
    The second term in the right-hand side can be dropped as $1 - 2\eta L - \frac{c L}{n\eta}= 1 - 2\eta L - 3\eta L\le 0$. In addition, $\rho\le \eta\mu$ and $\rho \le \frac{\tau}{3n}$, so the claim follows.
\end{proof}
Now we are ready to prove Lemma~\ref{lem:svrg_saga}.
\begin{customlem}{\ref{lem:svrg_saga}}
	In SVRG and SAGA, if $f_i$ is $L$-smooth and convex for all $i$, Assumption~\ref{as:method} is satisfied with $\eta_0 = \frac{1}{6L}$, $\omega = \frac{1}{3}$, $\rho = \frac{1}{3n}$ and 
	\begin{align*}
		\cM^t 
		= \frac{3\eta^2}{n}\sumin \EE\|\nabla f_i(u_i^t) - \nabla f_i(x^*)\|^2,
	\end{align*}
	where in SVRG $u_i^t=u^t$ is the reference point of the current loop, and in SAGA $u_i^t$ is the point whose gradient is stored in memory for function $f_i$. If $f$ is also strongly convex, then Assumption~\ref{as:method} holds with $\eta_0 = \frac{1}{5L}$, $\omega = 1$, $\rho = \frac{1}{3n}$ and the same $\cM^t$.
\end{customlem}
\begin{proof}
	Equation~\eqref{eq:saga_precise} gives immediately the second part of the claim. 
	
	Similarly, if $\eta\le \frac{1}{6L}$, from Equation~\ref{eq:saga_precise} we obtain by mentioning $1 - 2\eta L - \frac{c L}{n\eta} = 1 - 5\eta L \le \frac{1}{6}$ that
	\begin{align*}
        \EE \left[\|w^{t} - w^*\|^2 + \frac{c}{n}\sumkn \|\alpha_k^{t+1} - \alpha_k^*\|^2 \right]
        &\le \EE\left[\|x^t - x^*\|^2 - \frac{\eta}{3}D_f(x^t, x^*) + \frac{c}{n}\sumkn \|\alpha_i^t - \alpha_i^*\|^2\right].
    \end{align*}
\end{proof}
\subsection{Proof of Lemma~\ref{lem:sgd} (SGD)}\label{ap:sgd}
\begin{customlem}{\ref{lem:sgd}}
	Assume that at an optimum $x^*$ the variance of stochastic gradients is finite, i.e., $\sigma_*^2\eqdef \EE \|\nabla f_{\xi^t}(x^*)  - \nabla f(x^*)\|^2 < +\infty$. Then, SGD that terminates after at most $t_0$ iterations satisfies Assumption~\ref{as:method}(a) with $\eta_0=\frac{1}{4L}$, $\omega=1$ and $\rho=0$. If $f$ is strongly convex, it satisfies Assumption~\ref{as:method}(b) with $\eta_0 = \frac{1}{2L}$, $\omega=1$ and $\rho=0$. In both cases, sequence $\{\cM^t\}_{t=0}^{t_0}$ is given by
	\begin{align*}
		\cM^t = 2\eta^2(t_0 - t)\sigma_*^2.
	\end{align*}
\end{customlem}
\begin{proof}
	Clearly, we have
	\begin{align*}
		\EE \|w^t - w^*\|^2
		&= \EE\left[\|x^t - x^*\|^2 - 2\eta\<\nabla f(x^t) - \nabla f(x^*), x^t - x^*> + \eta^2 \|\nabla f_{\xi^t}(x^t)  - \nabla f(x^*)\|^2\right] \\
		&\overset{\eqref{eq:sgd_variance}}{\le} \EE\left[\|x^t - x^*\|^2 - 2\eta\<\nabla f(x^t) - \nabla f(x^*), x^t - x^*> + \eta^2\left(4L D_f(x^t, x^*) + 2\sigma_*^2\right)\right] \\
		&\overset{\eqref{eq:scal_prod_cvx}}{\le}\EE\left[(1 - \eta\mu)\|x^t - x^*\|^2 - 2\eta(1 - 2\eta L)D_f(x^t, x^*)\right] +2\eta^2 \sigma_*^2 .
	\end{align*}
	If $f$ is not strongly convex, then $\mu=0$ and by assuming $\eta \le \eta_0 = \frac{1}{4L}$ we get $1 - 2\eta L\ge \frac{1}{2}$ and
	\begin{align*}
		\EE\|w^t - w^*\|^2
		\le \EE\left[\|x^t - x^*\|^2 - \eta D_f(x^t, x^*)\right] + 2\eta^2 \sigma_*^2.
	\end{align*}
	In case $\mu=0$, by defining $\{\cM^{t}\}_{t=0}^{t_0}$ with recursion 
	\begin{align*}
		\cM^{t+1}
		= \cM^t - 2\eta^2\sigma_*^2,
	\end{align*}
	we can verify Assumption~\ref{as:method}(a) as long as $\cM^{t_0} = \cM^0 - 2t_0 \eta^2\sigma_*^2\ge 0$. This is the reason we choose $\cM^0=2t_0 \eta^2\sigma_*^2$.
	
	On the other hand, when $\mu>0$ and $\sigma_*=0$, it follows from $\eta\le \eta_0=\frac{1}{2L}$ that 
	\begin{align*}
		\EE \|w^t - w^*\|^2
		\le (1 - \eta\mu)\EE\|x^t - x^*\|^2.
	\end{align*}
\end{proof}

\clearpage

\section{Additional Experiments}\label{sec:add_experiments}
Here we want to see how changing $m$ and $n$ affects the comparison between SVRG with exact projection and decoupled SVRG with one stochastic projection. The problem that we consider is again $\ell_2$-regularized constrained linear regression. We took Gisette dataset from LIBSVM, whose dimension is $d=5000$, and used its first 1000 observations to construct $f$ and $g$. In particular, we split these observations into soft loss $f_i(x)=\frac{1}{2}\|a_i^\top x - b_i\|^2$ and hard constraints $g_j(x) = \ind_{\{x:a_j^\top x = b_j\}}$ with $n+m=1000$ and we considered three choices of $n$: 250, 500 and 750. To make sure that the constraints can be satisfied, we generated a random vector $x_0$ from normal distribution $\cN(0, \nicefrac{1}{\sqrt{d}})$ and set $b=\mA x_0$. In all cases, first part of data was used in $f$ and the rest in $g$. To better see the effect of changing $n$, we used fixed $\ell_2$ penalty of order $\nicefrac{1}{(n+m)}$ for all choices of $n$.

Computing the projection of a point onto the intersection of all constraints as at least as expensive as $m$ individual projections and we count it as such for SVRG. In practice it might be by orders of magnitude slower than this estimate for big matrices, but the advantage of our method can be seen even without taking it into account. On the other hand, to make the comparison fair in terms of computation trade-off, we use SVRG with minibatch 20 and our method with minibatch 1. The stepsize for both methods is $\nicefrac{1}{(2L)}$.
\begin{figure}[t]
     \centering
     \begin{subfigure}[t]{0.32\textwidth}
         \centering
         \includegraphics[width=1\textwidth]{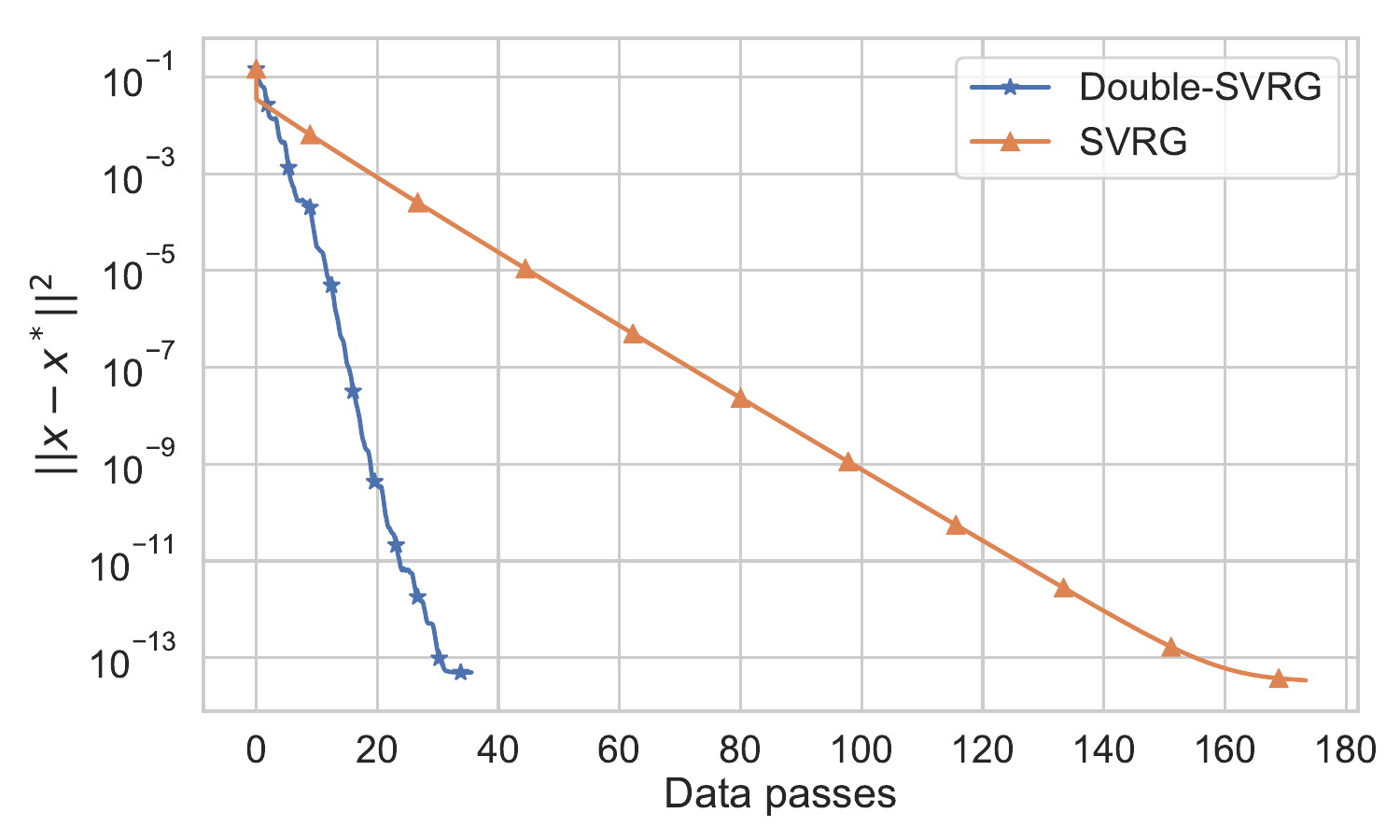}
         \caption{$m=100$, $n=900$}
     \end{subfigure}
     \begin{subfigure}[t]{0.32\textwidth}
         \centering
         \includegraphics[width=1\textwidth]{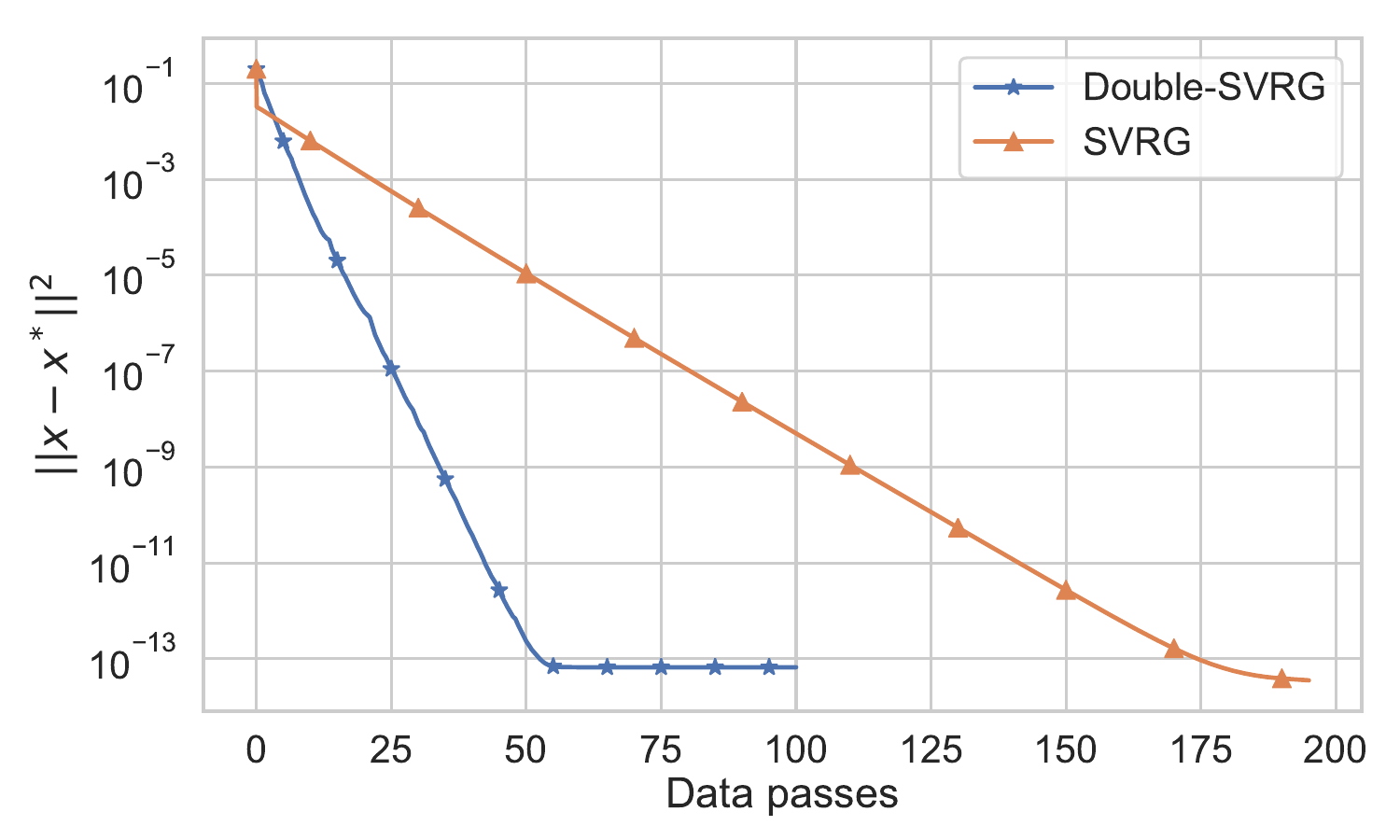}
         \caption{$m=200$, $n=800$}
     \end{subfigure}
     \begin{subfigure}[t]{0.32\textwidth}
         \centering
         \includegraphics[width=1\textwidth]{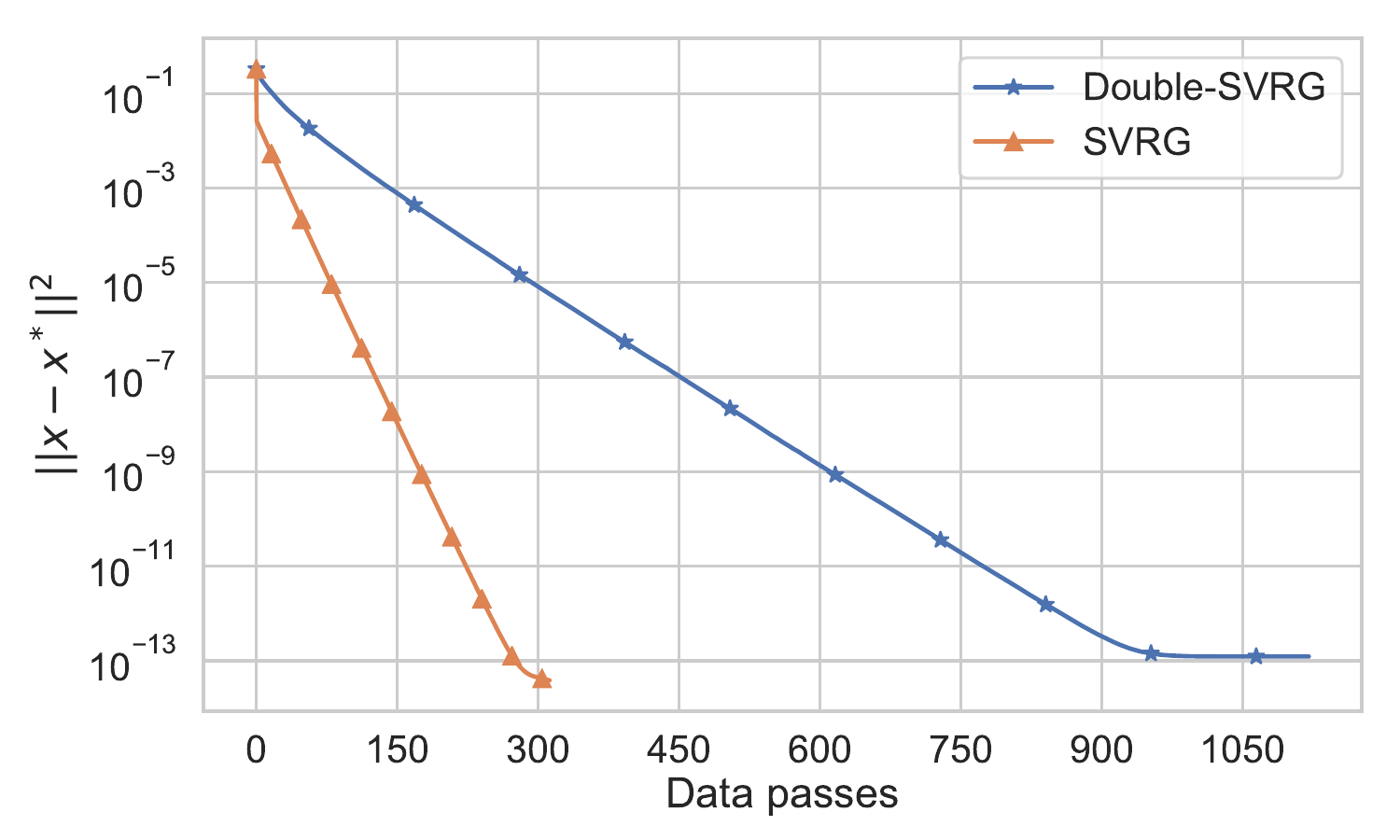}
         \caption{$m=500$, $n=500$}
     \end{subfigure}
     \\
     \begin{subfigure}[t]{0.32\textwidth}
         \centering
         \includegraphics[width=1\textwidth]{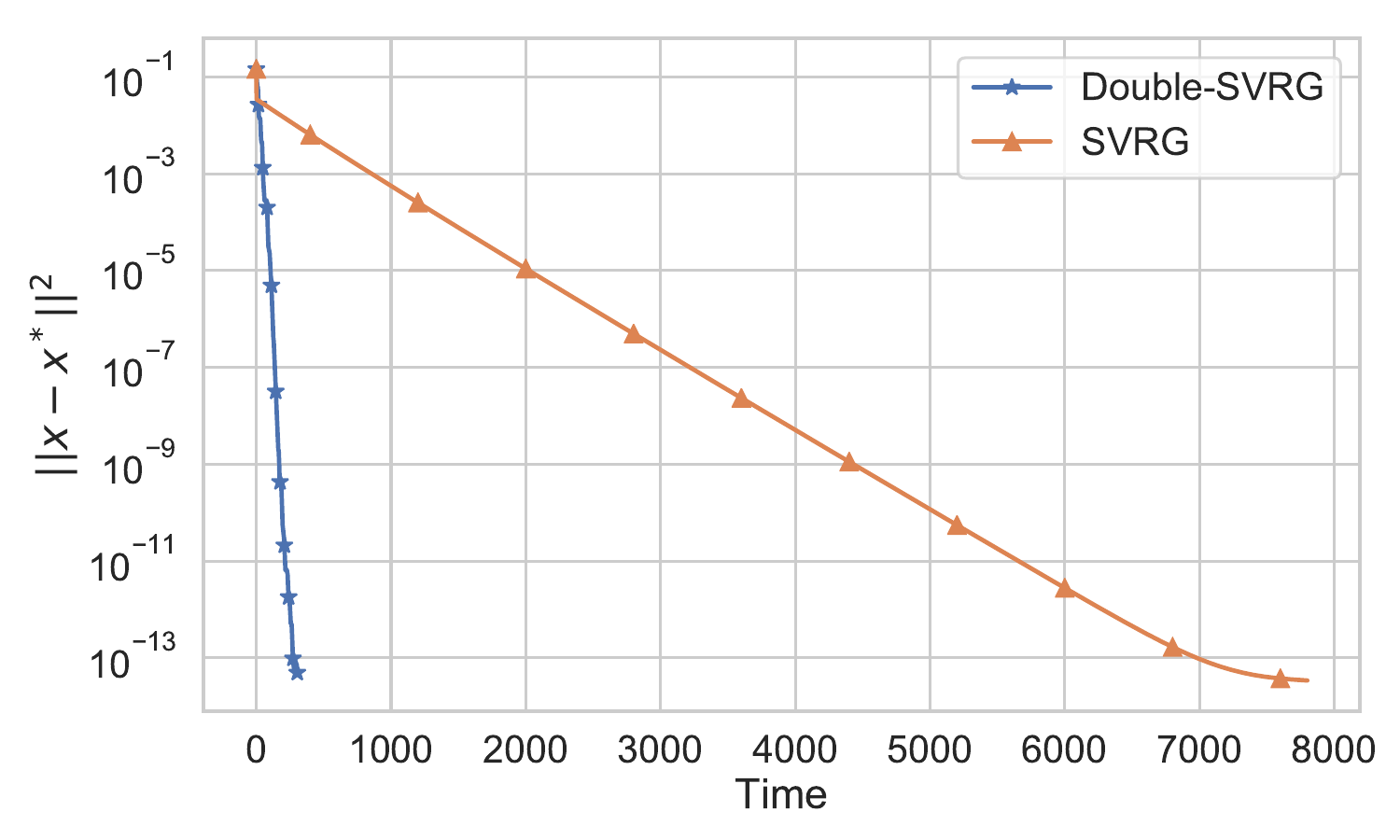}
         \caption{$m=100$, $n=900$}
     \end{subfigure}
     \begin{subfigure}[t]{0.32\textwidth}
         \centering
         \includegraphics[width=1\textwidth]{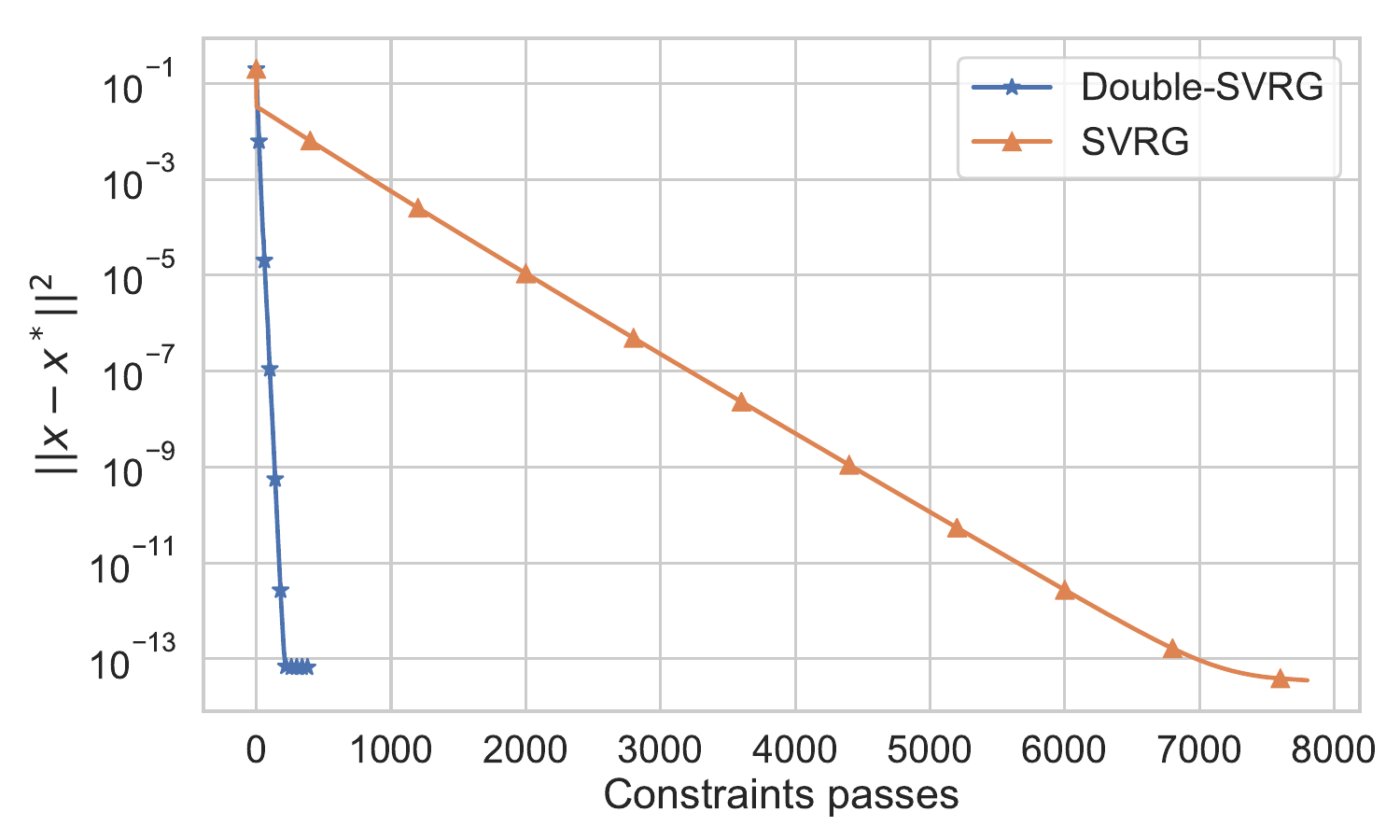}
         \caption{$m=200$, $n=800$}
     \end{subfigure}
     \begin{subfigure}[t]{0.32\textwidth}
         \centering
         \includegraphics[width=1\textwidth]{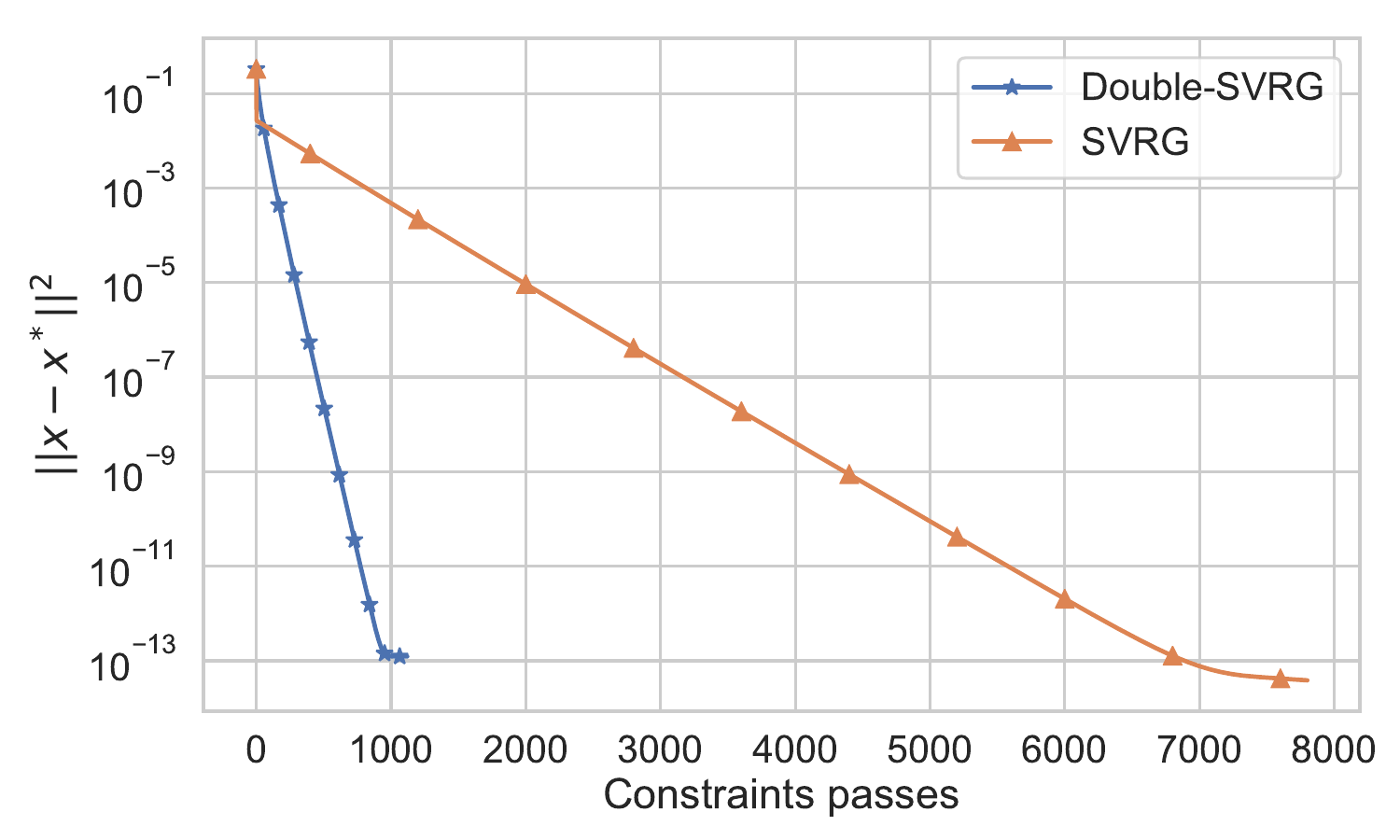}
         \caption{$m=500$, $n=500$}
     \end{subfigure}
     \caption{Comparison of SVRG with precise projection onto all constraints (labeled as 'SVRG') to our stochastic version of SVRG (labeled as 'Double-SVRG').}
     \label{fig:dif_m_and_n}
\end{figure}

As we can from Figure~\ref{fig:dif_m_and_n}, the trade-off between projections and gradients changes dramatically when $m$ increases. When $m=100$, which implies that the term corresponding to $\mA$ in the complexity is small, the difference is tremendous, partially because minibatching for SVRG improves only part of its complexity~\cite{gower2018stochastic}. In the setting $m=n=500$, we see that the number of data passes  taken by our method to solve the problem is a few times bigger than than that taken by SVRG. Clearly, this happens because the term related to $\mA$ becomes dominating in the complexity and SVRG uses $m=500$ times more constraints at each iteration than our method.

\clearpage

\section{Table of Key Notation}

\begin{table}[!h]
\caption{Key notation used in this paper.}
\centering
\begin{tabular}{|c|c|}
\hline
Objective function $F$ & $F \eqdef f + g + R$ \\
Domain of $F$ & $\dom F \eqdef \{x\;:\; F(x)<+\infty\} \neq \emptyset$ \\
Primal variable & $x\in \RR^d$ \\
Set of optimal solutions & $\cX^* \eqdef \{x^*\in \RR^d \;:\; F(x) \geq F(x^*) \; \forall x\in \RR^d\}$ (non-empty) \\
$f$ in finite-sum form & $f(x) = \frac{1}{n}\sum \limits_{i=1}^n f_i(x)$  ($f_i$ are differentiable and convex) \\
$f$ in expectation form & $f(x) = \EE_{\xi} f_\xi(x)$  ($f_\xi$ are differentiable and convex)\\
Standard Euclidean norm & $\|x\|\eqdef (\sum_{l=1}^d x_l^2 )^{1/2}$ \\
Gradient noise at optimum & $\sigma_*^2\eqdef \EE_\xi \|\nabla f_\xi(x^*) - \nabla f(x^*)\|^2$ \\
Smoothness constant of $f$ & $L$ \\
Strong convexity constant of $f$ & $\mu$ \\
Function $g_j$ &  $g_j:\RR^d\to \RR\cup \{+\infty\}$ (proper, closed, convex, proximable)\\
Function $g_j$ in a structured form & $g_j(x) = \phi_j(\mA_j^\top x)$, $\mA_j\in \RR^{d\times d_j}$ \\
Function $\phi_j$ & $\phi_j:\RR^{d_j}\to \RR\cup \{+\infty\}$ (proper, closed, convex, proximable)\\
Function $g$ & $g(x) \eqdef \frac{1}{m}\sum \limits_{j=1}^m g_j(x)$  (proper, closed, convex)\\
Function $R$ & $R(x):\RR^d \to \RR\cup \{+\infty\}$  (proper, closed, convex, proximable) \\
\hline
Primal iterates & $x^t \in \RR^d$ \\
Estimator of $\nabla f(x^t)$ & $v^t$ \\
Dual variables & $y_1^t,\dots,y_j^t \in \RR^d$ \\
Probability of selecting index $j$ & $p_j$ \\
Stepsize associated with $f$ and $R$ & $\eta$ \\
Stepsize associated with $g_j$ & $\eta_j = \frac{\eta}{m p_j}$ \\
Lyapunov function &   $\cL^t
        \eqdef  \EE\|x^t - x^*\|^2 + \cM^t + \cY^t$ \\
$\cY^t$ & $ \cY^t
	\eqdef (1+\gamma)\sumkm \eta_j^2\EE\|y_j^{t} - y_j^*\|^2$\\      
Smoothness constant of $g_j$ &  $L_j\in\RR\cup \{+\infty\}$ ($L_j=+\infty$ if $g_j$ is non-smooth)\\
Parameter $\gamma$	& $\gamma \eqdef \min_{j=1,\dots,m} \frac{1}{\eta_j L_j}$ ($\gamma=0$ if any $g_j$ is non-smooth)  \\
\hline
Subdifferential of $R$ & $\partial R(x)\eqdef \{s \;:\; R(y) \geq R(x) + \langle s, y-x \rangle\}$, $x\in \dom R$ \\
Proximal operator of function $R$ & $\prox_{\eta R}(x) \eqdef \argmin_{u} \left\{R(u) + \frac{1}{2\eta}\|u - x\|^2 \right\}$ \\ 
Proximal operator of function $g_j$ & $\prox_{\eta_j g_j}(x) \eqdef \argmin_{u} \left\{g_j(u) + \frac{1}{2\eta_j}\|u - x\|^2 \right\}$ \\ 
Characteristic function of a set $\cC$ & $\ind_{\cC}(x)\eqdef \begin{cases}0 & x\in \cC \\ +\infty & x\notin \cC\end{cases}$ \\
Projection onto set $\cC$ & $\Pi_{\cC}(x) \eqdef  \prox_{\ind_{\cC}}(x)= \argmin_{u\in \cC} \|u-x\|$ \\
Bregman divergence of $f$ & $D_f(x, y) \eqdef f(x) - f(y) - \<\nabla f(x), x - y> $ \\
\hline
\end{tabular}
\label{tbl:notation}
\end{table}

\end{document}